%% file: Latest_Version.tex
\input{math_commands.tex}

\documentclass[12pt]{amsart}
\usepackage{amsmath,amsfonts,bm}
\usepackage{amsmath}
\usepackage{amssymb}
\usepackage{mathtools}

\usepackage{algorithm}
\usepackage[noend]{algpseudocode}

\usepackage{bm}
\usepackage{enumitem}
\usepackage{graphicx}
\usepackage{subcaption}
\usepackage{caption}
\usepackage{float}
\usepackage{hyperref}

\usepackage{parskip}
\usepackage{xcolor}

\usepackage{multirow}
\usepackage{booktabs}
\usepackage{tabularray}
\allowdisplaybreaks

\def\red{\textcolor{red}}
\def\blue{\textcolor{blue}}

\newtheorem{theorem}{Theorem}
\newtheorem{lemma}[theorem]{Lemma}

\theoremstyle{definition}

\theoremstyle{remark}

\numberwithin{equation}{section}

\allowdisplaybreaks

\numberwithin{equation}{section}
\renewcommand{\theequation}{\arabic{section}.\arabic{equation}}

\begin{document}
\title[inf--sup] 
{inf--sup NEURAL NETWORKS FOR HIGH DIMENSIONAL PDEs}
\author{Ziren Chen}
\address{Iowa State University, Department of Mathematics, Ames, IA 50011} 
\author{Hailiang Liu}
\address{Iowa State University, Department of Mathematics, Ames, IA 50011} \email{zc2m@iastate.edu, hliu@iastate.edu}


\begin{abstract}
Solving partial differential equations (PDEs) in high dimensions remains challenging due to the curse of dimensionality. We propose a neural-network-based framework that reformulates PDEs as inf--sup optimization problems through the introduction of a Lagrange multiplier. The primal solution and the associated Lagrange multiplier are parameterized by two networks and are computed via an iterative saddle-point optimization procedure. We prove the theoretical equivalence between the proposed optimization formulation and the original PDE problem, and we derive rigorous error estimates that quantify the total approximation error in terms of the network approximation error, statistical (sampling) error, and optimization error. Numerical experiments demonstrate the accuracy, stability, and efficiency of the proposed method for solving high-dimensional PDEs.
\end{abstract}

\maketitle



\section{Introduction}
Partial differential equations (PDEs) play a central role in numerous areas of mathematics, physics, and engineering.  They provide a rigorous framework for modeling systems governed by physical laws and have found applications in diverse domains including  fluid dynamics, electromagnetism and material science. Traditional numerical methods for solving them  have achieved  high accuracy in low-dimensional settings but rely heavily on discretizing the state space. As the dimensionality increases, the number of grid points required for accurate approximations grows exponentially, leading to prohibitive computational costs \cite{li_liu_2002, tadmor_2012}.

Recent advances in neural network-based methods have emerged as promising alternatives for PDE solvers, leveraging the representational power of deep neural networks to approximate complex solutions in a mesh-free, data-efficient manner \cite{sirignano_dgm_2018, e_deep_2018, raissi_physics-informed_2019, zang_weak_2020, li_fourier_2021, lu_deeponet_2021,  liu_natural_2024, huo_inf-sup_2024}. Among these approaches, physics-informed neural networks (PINNs) \cite{raissi_physics-informed_2019} and their variants \cite{ kharazmi_zhang_karniadakis_2019, jagtap_karniadakis_2020, yu_gradient-enhanced_2022,mcclenny_self-adaptive_2023, de_ryck_mishra_molinaro_2024} have gained considerable attention for incorporating  physical laws directly into the training objective. Nevertheless, PINNs may exhibit stability issues and limited generalization in certain applications \cite{krishnapriyan_characterizing_2021,wang_understanding_2021, wang_when_2022}. To address these challenges,  weak-residual formulations have been explored, see e.g., \cite{kharazmi_zhang_karniadakis_2019, zang_weak_2020, huo_inf-sup_2024}. In particular, Inf-SupNet \cite{huo_inf-sup_2024} reformulate the PDE as a constrained optimization problem, where the boundary data mismatch is minimized subject to the governing equation. This formulation is then recast as an inf--sup problem through  the introduction of a Lagrange multiplier. Building on this perspective, the present work  develops a more general inf--sup neural framework applicable to a broader class of PDEs.

In this work, we introduce InfsupNet, a neural framework, for solving a broad range of PDEs by reformulating the solution process as an inf-sup optimization problem.
InfsupNet builds upon a variational characterization of PDEs, enabling the learning of solutions through a game-theoretic perspective. By expressing the PDE problem as a saddle-point optimization between primal and dual functions, this approach naturally leads to improved stability and convergence properties. Crucially, it avoids direct residual minimization, which can obscure the true dynamics of the solution process and introduce training pathologies. 

{From a theoretical standpoint, error analysis for inf--sup formulations is inherently more involved. For well-posed PDE problems, stability guarantees that the error is controlled by the residual. This connection has been explored in \cite{mishra_estimates_2023, de_ryck_error_2022} to derive generalization error bounds for PINNs. However, further relating the residual error to the saddle-point loss requires a different approach, as introduced in \cite{huo_inf-sup_2024} through InfSupNet for elliptic problems.  
In this work, building on the framework of  \cite{huo_inf-sup_2024}, we establish error bounds for linear well-posed  PDEs. Specifically, we decompose the total error into three components: the sampling error, the neural network approximation error, and the optimization gap. Each component is explicitly characterized and quantitatively estimated, yielding a transparent interpretation  of their respective  contributions. Furthermore, we derive upper bounds for these components in the case of linear convection–diffusion equations.}

\subsection{Related work.} 
In the following section, we briefly review representative neural PDE solvers, categorized by their underlying formulations, and present related results on error bounds.

\textbf{Residual-minimization formulations.} Early neural approaches \cite{MN1994,LAD1998} based on residual minimization were typically limited to a single  hidden layer, with boundary conditions imposed manually via a change of variables. More recent methods adopt deeper architectures, such as the Deep Galerkin Method (DGM) \cite{sirignano_dgm_2018}, physics-informed neural networks (PINNs) \cite{raissi_physics-informed_2019}. Both approaches typically minimize a combination of  the squared PDE residuals and data mismatch, which can implicitly lead to higher-order equations and difficulties when the solution exhibits low regularity. Moreover, they often require careful manual tuning to balance residual and boundary losses. To address these issues, several extensions have been proposed, including gradient-enhanced PINNs (gPINNs) \cite{yu_gradient-enhanced_2022}, variational PINNs (VPINNs) \cite{kharazmi_zhang_karniadakis_2019}, and self-adaptive PINNs \cite{mcclenny_self-adaptive_2023} with trainable loss weights. 

\textbf{Variational/Ritz formulations.} 
The Deep Ritz Method (DRM) \cite{e_deep_2018} minimizes a variational (energy) functional associated with the underlying PDE, 
rather than the residual itself. This formulation is particularly well suited to problems that admit a variational structure, such as elliptic equations, and often yields improved stability and convergence compared to residual-based approaches like PINNs when a well-defined energy exists. {Further developments along this line include DRM with adaptive quadrature \cite{liu_cai_ramani_2023},  DRM with Fourier feature mapping \cite{mema_wang_knap_2025}, iterative DRM \cite{hu_jin_wang_2025}, and DRM-PINN hybrid approaches \cite{zhou_wang_deng_li_2025}.}

{\textbf{Weak-residual formulations.} Methods such as VPINNs \cite{kharazmi_zhang_karniadakis_2019}, Weak Adversarial Networks (WANs) \cite{zang_weak_2020}, Inf-SupNet \cite{huo_inf-sup_2024}  are closely related through their use of weak or saddle-point formulations of PDEs.  VPINNs and WANs are both grounded in the weak form of the governing PDEs. VPINNs evaluate this weak residual against predefined test spaces (e.g.,piecewise polynomial functions) using numerical quadrature, whereas WANs parameterize both the primal solution and the test functions with neural networks, leading to a min–max adversarial optimization that minimizes an operator norm. In contrast, other approaches explicitly construct saddle-point problems. Inf-SupNet \cite{huo_inf-sup_2024} enforces the PDE as a constraint by introducing a Lagrange multiplier, thereby reformulating the problem as an inf-sup optimization. PDHG-based neural solvers \cite{liu_natural_2024} explore suitable preconditioning toward  a natural primal-dual gradient ascent--descent scheme with improved computational efficiency.}

\textbf{Operator learning. }A newer class of neural methods shifts the focus from solving a single equation to learning solution operators that generalize across families of parameters or initial conditions; for instance,  the Fourier Neural Operator (FNO) \cite{li_fourier_2021} and DeepONet \cite{lu_deeponet_2021}. While enabling fast inference, they typically require large datasets of precomputed solutions for training. {To address this limitation, Wang et al. (Wang, Wang, and Perdikaris \cite{wang_wang_perdikaris_2021} Wang and Perdikaris \cite{wang_perdikaris_2023}) introduced the physics-informed DeepONet, in which the loss function is constructed solely 
from physical constraints, thus making the model self-supervised. The approach is further extended through  transfer learning techniques \cite{xu_lu_wang_2023}.}




{\textbf{Neural solvers for density evolution.} Mass-conserving PDEs, 
such as the Fokker-Planck equation,  Wasserstein gradient flows, and related mean-field models, are typically subject to structural constraints, including non-negativity and mass conservation of the evolved density. The neural PDE solvers discussed above do not directly enforce these constraints, making their application to such problems challenging. To address this issue, several neural approaches have been developed. One line of work applies deep learning to the time-discretized JKO scheme \cite{JKO98} for solving Wasserstein gradient flows; see e.g., \cite{MKel21, ASM21, FTC21}. These methods transport the initial density through a sequence of neural networks and hence avoid space discretization. However, the total training cost may scale quadratically with the number of JKO steps. Another line of work develops parametric approximations based on push-forward maps, neural networks, and the Wasserstein metric, for both Fokker-Planck equations \cite{LLZZ19, LLZZ22} and for Wasserstein Hamilton flows \cite{WLYZ25}. More recently, several works have employed  velocity-matching ideas to construct flows driven by learned velocity fields.
These include score-based methods,  which parameterize the probability flow through the score function for Fokker--Planck equations \cite{boffi_probability_2023, lu_score-based_2024, ZOL25a, hu_zhang_karniadakis_kawaguchi_2025}, score-based neural ODEs for mean field control problems \cite{ZOL25b}, and self-consistent velocity matching methods \cite{shen_self-consistency_2022, li_self-consistent_2023, shen_wang_2023}.}

\textbf{Error bounds.} {Recent work has made substantial progress in establishing  error bounds for neural network-based PDE solvers. For PINNs, Mishra and Molinaro \cite{mishra_estimates_2023} developed an a posteriori framework that leverages the stability properties of the underlying PDE to relate generalization error to both training error and sampling complexity; this approach was further refined by De Ryck and Mishra \cite{de_ryck_error_2022} for linear parabolic equations. For DRM,  Lu, Lu and Wang \cite{lu_priori_2021} derived an a priori error bounds based on the variational structure of elliptic PDEs, while  Minakowski and Richter \cite{minakowski_priori_2023}  established finite element-type estimates incorporating approximation and quadrature errors.}  

{Despite these advances, error analysis for weak-residual formulations remains relatively limited. For WANs, Bertoluzza, Burman, and He \cite{bertoluzza_burman_he_2024} showed that the error can be controlled by the approximation properties of  the trial space  under suitable stability conditions, while Jiao et al. \cite{jiao_convergence_2023} incorporated statistical errors, though with rates deteriorating in the increasing spatial dimension. 
More recently, Liu, Osher, and Li \cite{liu_natural_2024} established convergence results for time-continuous formulations.} 
{In contrast, \cite{huo_inf-sup_2024} introduced a novel error decomposition for InfSupNet applied to  elliptic problems, enabling   simultaneous control of both the primal variable and the Lagrange multiplier. 
The present analysis can be viewed as a natural extension of this framework}.

\subsection{Problem Formulation and Inf--Sup Framework}
We extend our prior work \cite{huo_inf-sup_2024} to a broader class of PDEs. Let $Q \subset \mathbb{R}^d$ ($d>1$) be a bounded domain with boundary $\partial Q$, and let $\Gamma \subset \partial Q$ denote the portion where boundary conditions are imposed. Let $X$ be a Banach space on $Q$, and let $Z$ and $P$ be Hilbert and Banach spaces defined on $Q$ and $\Gamma$, respectively. Given $f\in Z$ and $g\in P$, consider
\begin{align}\label{eq01}
\begin{split}
\mathcal{F}[\xi,u] &= f \quad \text{in } Q, \\
\mathcal{B}u &= g \quad \text{on } \Gamma, 
\end{split}
\end{align}
where $\mathcal{F}: Q \times X \to Z$ is a differential operator and $\mathcal{B}: X \to P$ is a boundary operator. Assuming the problem is well-posed, we reformulate is as the constrained minimizaiton problem 
\begin{align}\label{eq02}
\inf_{u\in U}
\frac12\,\mathrm{dist}(\mathcal{B}u,g)^2,
\qquad
U=\{u\in X;\ \mathcal{F}[\xi,u]=f\text{ in }Q\},
\end{align}
where $\mathrm{dist}(\cdot,\cdot)$ denotes a suitable boundary mismatch functional. Introducing a Lagrange multiplier $v\in Z^*$ leads to the equivalent inf--sup formulation
\begin{align}\label{eq03}
\inf_{u\in X}\sup_{v\in Z^*}
\mathcal{L}(u,v)
:=
\frac12\,\mathrm{dist}(\mathcal{B}u,g)^2
+\langle\mathcal{F}[\xi,u]-f,v\rangle_{Z,Z^*}.
\end{align}
This formulation provides a unified optimization framework that extends InfSupNet to a broader range of PDEs. For numerical approximation, we introduce neural network hypothesis spaces $\mathcal{U}_\theta \subset X$ and $\mathcal{V}_\tau \subset Z^*$ for the primal and dual variables. Let $\widehat{\mathcal{L}}$ denote a discrete approximation of $\mathcal{L}$ obtained via sampling in $Q$ and $\Gamma$. The resulting parameterized 
min--max problem is
\begin{align}\label{eq04}
\min_{u_\theta \in \mathcal{U}_\theta}
\max_{v_\tau \in \mathcal{V}_\tau}
\widehat{\mathcal{L}}(u_\theta,v_\tau).
\end{align}
Our goal is to compute a saddle point $(\theta^*,\tau^*)$ such that 
\begin{align}\label{eq05}
\widehat{\mathcal{L}}(u_{\theta^*},v_{\tau^*})
=
\min_{u_\theta}\max_{v_\tau}
\widehat{\mathcal{L}}(u_\theta,v_\tau).
\end{align}
Algorithmic details are provided in Section~3.

\subsection{Theoretical Foundations}

The development of neural PDE solvers relies on several theoretical components: sampling-based discretization, function-space approximation, and optimization. 

Sampling and numerical integration are central to both efficiency and accuracy. Unlike classical mesh-based methods, many neural approaches rely on random or quasi-random sampling of the computational domain, which is particularly well suited for high-dimensional problems. Monte Carlo and quadrature-based schemes have been extensively studied in the context of mesh-free PDE solvers, including their effects on stability and accuracy \cite{caflisch_monte_1998}. Numerical integration for fractional Sobolev norms, relevant in the boundary mismatch terms in our formulation, has been analyzed in \cite{bonito_convergence_2025,mishra_priori_2025}. These results provide theoretical justification for the sampling strategy adopted in our empirical loss.

Understanding the approximation properties of neural networks in Sobolev-type spaces is crucial for characterizing the hypothesis classes $\mathcal{U}_\theta$ and $\mathcal{V}_\tau$. Classical universal approximation results \cite{cybenko_approximation_1989} establish density in continuous function spaces, while more recent work extends these guarantees to Sobolev and Sobolev–time spaces \cite{de_ryck_approximation_2021,abdeljawad_approximations_2022}. These results yield quantitative error bounds that are directly relevant to our approximation error analysis. 

Finally, since our formulation leads to a min--max problem, the convergence of saddle-point algorithms play a key role. Gradient-based methods for convex--concave problems are well understood \cite{nemirovski_prox-method_2004}, and recent advances address weaker convexity assumptions and certain nonconvex settings \cite{lin_gradient_2020,jin_what_2020,razaviyayn_nonconvex_2020}. These works guide the algorithmic design and convergence discussion presented in Section~3.

\subsection{Main Contributions}

The main contributions of this paper are as follows:
\begin{itemize}
    \item We extend the inf--sup framework introduced in \cite{huo_inf-sup_2024} to a neural network-based solver applicable to a broad class of PDEs, including both stationary and time-dependent problems.
    \item We establish the theoretical equivalence between the proposed inf--sup formulation and the original PDE problem under appropriate well-posedness assumptions.
    \item For time-dependent linear convection-diffusion equations, we develop a rigorous inf--sup error analysis for neural network approximations of PDEs. In particular, we prove that the saddle-point error admits a sharp decomposition into three fundamental components: the neural network approximation error $I_{\mathrm{NN}}$, the Monte Carlo sampling error $I_{\mathrm{MC}}$, and the optimization gap $I_{\mathrm{GP}}$. Specifically, for the iterates $(u_n,v_n)$ generated by the proposed inf--sup training algorithm, we establish error bounds of the form
    \[
    \|u_n-u\|_X^2 + \|v_n-v\|_{Z^*}^2 \;\le\; C\,(I_{\mathrm{NN}} + I_{\mathrm{MC}} + I_{\mathrm{GP}}),
    \]
    where each term is explicitly characterized and quantitatively estimated.
    \item We validate the proposed method through numerical experiments, demonstrating its accuracy and stability for representative PDE problems.
\end{itemize}

\subsection{Organization of the Paper}

The remainder of this paper is organized as follows. Section~2 introduces the necessary mathematical preliminaries and formulates the inf--sup framework for general PDE problems. Section~3 describes the neural network architecture and the corresponding training algorithm based on the inf--sup formulation. Section~4 presents a rigorous error decomposition and establishes theoretical guarantees for the proposed method. In Section~5, we derive explicit upper bounds for each component of the approximation error, including results under reduced regularity assumptions. Section~6 reports numerical experiments that demonstrate the accuracy and stability of the method. Finally, Section~7 concludes the paper and outlines directions for future research.


\section{Problem Setup and Reformulation}  

In this section, we develop an abstract analytical framework for the class of PDE problems under consideration.  We begin by  formulating the PDE as a constrained optimization problem. Subsequently,  by introducing an appropriate Lagrange multiplier, we recast the problem into an equivalent inf–sup (saddle-point) formulation. We rigorously prove the equivalence between the original PDE problem and the constrained optimization problem. Moreover, we establish the solutions to the inf--sup formulation correspond exactly to solutions of the underlying PDE, thereby demonstrating the full equivalence of all three formulations.

\subsection{Representative PDE Examples}

This general formulation \eqref{eq01} covers a wide range of PDE problems, including both time-independent problems where  $\xi = x$ represents spatial variables, and time-dependent problems where $\xi = (\mathbf{x},t)$ represents spatial and temporal variables. 

Representative examples include:
\begin{itemize}
\item \textbf{Poisson equation} with Dirichlet boundary conditions (steady-state case):
\[
\mathcal{F}[x, u] = -\Delta u = f(x), \quad \text{in } \Omega ,
\]
with boundary operator 
\[
\mathcal{B}u(x) = u(x) = g(x), \quad \text{on } \Gamma = \partial \Omega.
\]
\item \textbf{Convection-diffusion equation}(time-dependent):
  \[
  \mathcal{F}[x,t, u] = u_t + \mathbf{b} \cdot \nabla u - \varepsilon \Delta u = f(x,t), \quad \text{in } \Omega \times (0, T),
  \]
  with boundary and initial conditions:
   \[
    \mathcal{B}u(x, t)  = u(x,t) =
    \begin{cases}
       g(x,t), & \quad \text{on } \partial \Omega \times (0, T), \\
        h(x), & \quad \text{in } \Omega\times\{t = 0\}.
    \end{cases}
    \]
  \item \textbf{Reaction-diffusion equation:} (time-dependent):
  \[
  \mathcal{F}[\xi, u] = u_t - d \Delta u - k u^2 = 0, \quad \text{in } \Omega \times (0, T),
  \]
  with boundary and initial conditions:
   \[
    \mathcal{B}u(x, t)  = u(x,t) =
    \begin{cases}
       g(x,t), & \quad \text{on } \partial \Omega \times (0, T), \\
        h(x), & \quad \text{in } \Omega\times\{t = 0\}.
    \end{cases}
    \]
\end{itemize}

\subsection{Problem Reformulation and Equivalence} 
To solve the PDE problem \eqref{eq01}, we reformulate it as a constraint optimization problem \eqref{eq02}, in which the distance functional $\mathrm{dist}(\cdot,\cdot)$ is a metric. 

We now establish the equivalence between the original PDE problem \eqref{eq01} and the constrained optimization problem \eqref{eq02}.  
\begin{theorem}\label{thm:1}
Assume the PDE problem \eqref{eq01} is well-posed. Then $u \in X$ solves \eqref{eq01} if and only if $u \in X$ solves the constraint optimization problem \eqref{eq02}.
\end{theorem}
\begin{proof}
$(\Longrightarrow)$
Suppose $u\in X$ solve the PDE problem \eqref{eq01}. Then 
$u\in U$ and $Bu=g$, so  
\begin{align*}
\mathrm{dist}(\mathcal{B}u,g)^2 = 0.
\end{align*}
Since the objective is non-negative for all $w\in U$, this implies that $u$ minimizes the objective and solves \eqref{eq02}.

$(\Longleftarrow)$
Conversely, suppose $u$ solves the optimization problem \eqref{eq02}. Then 
\begin{align}\label{eq08}
    \mathrm{dist}(\mathcal{B}u,g)^2 \le \mathrm{dist}(\mathcal{B}w,g)^2 \quad \forall w\in U.
\end{align}
Let $u^* \in X$ be the unique solution to the well-posed PDE problem \eqref{eq01}. Then $u^* \in U$ and $\mathcal{B}u^*=g$, so 
$$
\mathrm{dist}(\mathcal{B}u,g)^2 \le \mathrm{dist}(\mathcal{B}u^*,g)^2 = 0 \Rightarrow \mathrm{dist}(\mathcal{B}u,g)^2=0.
$$
Hence $Bu=g$, and since $u\in U$, we have $F[\xi, u]=f$. 
By uniqueness, $u = u^*$, and thus $u$ solves \eqref{eq01}.
\end{proof}

To eliminate the explicit PDE constraint in~\eqref{eq02}, we introduce a Lagrange multiplier $v \in Z^*$, where $Z^*$ denotes the dual space of $Z$. This yields the inf--sup (saddle-point) optimization problem \eqref{eq03}.

The following theorem establishes the equivalence between the saddle points of ~\eqref{eq03} and the solutions of the original PDE problem ~\eqref{eq01}. 
\begin{theorem}
Assume the PDE problem \eqref{eq01} is well-posed. Then the following statements hold:

(1) If $u \in X$ is the unique solution to the PDE problem \eqref{eq01}, then $(u,v=0) \in X\times Z^*$ is a saddle point of $\mathcal{L}$ defined in \eqref{eq03}. That is,
$\mathcal{L}(u,\tilde v) \leq \mathcal{L}(u,0) \leq \mathcal{L}(\tilde u, 0) \quad \forall \tilde u \in X,\, \forall \tilde v \in Z^*$.

(2) Conversely, if $(u, v) \in X \times Z^* $ is a saddle point of $\mathcal{L}(u,v)$ as defined in \eqref{eq03}, then $v = 0$ and $u\in X$ is the solution to the PDE problem \eqref{eq01}.
\end{theorem}
\begin{proof}
(1) Suppose $u \in X$ is the unique solution to the PDE problem \eqref{eq01}. Then, $\forall \tilde{u} \in X$, we have 
\[
\mathcal{L}(\tilde{u}, 0) = \frac{1}{2}\mathrm{dist}(\mathcal{B}\tilde{u},g)^2 + \langle F - f, 0\rangle_{Z,Z^*} = \frac{1}{2}\mathrm{dist}(\mathcal{B}\tilde{u},g)^2 \ge 0 = \mathcal{L}(u, 0).
\]
Similarly, $\forall \tilde v \in Z^*$, we have: 
\[
\mathcal{L}(u, \tilde{v}) = \frac{1}{2}\mathrm{dist}(\mathcal{B}u,g)^2 + \langle F - f, \tilde{v}\rangle_{Z,Z^*} = 0,
\]
since $u$ satisfies the PDE constraint and thus $Bu=g$ and $F=f$. Therefore, 
$$
\mathcal{L}(u, \tilde{v})=0=\mathcal{L}(u, 0),
$$
which confirms that $(u,0)$ is a saddle point of $\mathcal{L}$.

(2) Conversely, suppose $(u, v) \in X \times Z^* $ is a saddle point of $\mathcal{L}$, i.e.,
$$
\mathcal{L}(u,v) = \inf_{w \in X} \sup_{q \in Z^*} \mathcal{L}(w,q) =  \sup_{q \in Z^*} \inf_{w \in X} \mathcal{L}(w,q),
$$
where
$$
\mathcal{L}(w, q) = \frac{1}{2}\mathrm{dist}(\mathcal{B}w,g)^2 + \langle F - f, q\rangle_{Z,Z^*}.
$$
If $F - f \neq 0$, then the second term $\langle F - f, q\rangle_{Z,Z^*}$ can become arbitrarily large in magnitude due to the unbounded nature of $q \in Z^*$. Therefore, 
\[
\sup_{q \in Z^*} \{\mathcal{L}(w,q)\}  =
\begin{cases}
   \frac{1}{2}\mathrm{dist}(\mathcal{B}w,g)^2, & \text{if } F - f = 0, \\
    \infty, & \text{Otherwise}.
\end{cases}
\]
It is clear that the infimum is achieved only when $\mathcal{F}[\xi, w] = f(\xi) \text{ in } Q$, i.e., when $w \in U$. Thus, problem \eqref{eq03} reduces to:
\[
\inf_{w \in U}\Big\{\frac{1}{2}(\mathrm{dist}(\mathcal{B}w,g)^2\Big\},
\] 
where $U = \{w\in X;\  \mathcal{F}[\xi, w] = f (\xi) \text{ in } Q\}$. This is precisely the constraint optimization problem defined in \eqref{eq02}. By Theorem~\ref{thm:1},  we conclude that $u$ is the solution to the PDE problem \eqref{eq01}.

Next, we  show that $v = 0$. For any $q \in Z^*$ with $q \neq 0$, there exists $ h \in Z$ such that $\langle h, q\rangle_{Z,Z^*} < 0$. Then there exists a unique $\tilde{w} \in X$ such that
\begin{align*}
\mathcal{F}[\xi, \tilde{w}] - f &= h \quad \text{in } Q, \\
\mathcal{B}\tilde{w} &= g \quad \text{on } \Gamma.
\end{align*}
Therefore,
\[
\begin{cases}
   \inf_{w \in X} \{\mathcal{L}(w,q)\}  = 0, & \text{if } q = 0, \\
    \inf_{w \in X} \{\mathcal{L}(w,q)\}  \le \mathcal{L}(\tilde{w},q) = \langle h, q\rangle_{Z,Z^*} < 0, & \text{if } q \neq 0.
\end{cases}
\]
Hence, the supremum over $q\in Z^*$ 
is achieved only when $q = 0$, which implies  $v=0$.  Furthermore, this shows that $(u,0)$ is the unique saddle point of $\mathcal{L}$.
\end{proof}



\section{Algorithm}

\subsection{Neural Network Approximation and Inf--Sup Networks}

To approximate the primal solution $u$ and the dual variable $v$ in the inf--sup formulation \eqref{eq03}, we introduce neural network hypothesis spaces $\mathcal{U}_\theta \subset X$ and $\mathcal{V}_\tau \subset Z^*$ parameterized by trainable parameters $\theta \in \mathbb{R}^p$ and $\tau \in \mathbb{R}^q$, respectively. The primal network $u_\theta \in \mathcal{U}_\theta$ approximates the PDE solution $u \in X$, while the dual network $v_\tau \in \mathcal{V}_\tau$ approximates the Lagrange multiplier $v \in Z^*$.

Embedding the neural networks directly into the function spaces $X$ and $Z^*$ is essential, as the loss functional involves differential operators acting on $u$ as well as duality pairings with $v$. 
Classical approximation theory ensures that sufficiently wide and deep feedforward neural networks achieve arbitrary accuracy in approximating functions in Sobolev spaces; see, for example, \cite{de_ryck_approximation_2021,abdeljawad_approximations_2022} for recent developments on approximation in Sobolev norms. Consequently, 
the expressive capacity of $\mathcal{U}_\theta$ and $\mathcal{V}_\tau$ is governed  by the network architecture --  specifically, its depth, width, and choice of activation functions.


Under this  parameterization, the inf--sup problem \eqref{eq03} is approximated by the finite-dimensional min--max problem
\begin{align}\label{eq10}
\min_{u_\theta \in \mathcal{U}_\theta}
\max_{v_\tau \in \mathcal{V}_\tau}
\mathcal{L}(u_\theta,v_\tau),
\end{align}
which can be solved numerically using saddle-point optimization algorithms.


\subsection{Sampling and Empirical Loss}

To evaluate the loss function numerically, we approximate all integrals using sampled points. Interior and boundary contributions are handled  separately, reflecting the different norms involved in the loss functional. Let $\{\xi_n^{\mathrm{int}}\}_{n=1}^{N_{\mathrm{int}}} \subset Q$ and $\{\xi_n^{\mathrm{bd}}\}_{n=1}^{N_{\mathrm{bd}}} \subset \Gamma$ denote  the interior and boundary sampling points, respectively. For example, suppose  
$$\mathrm{dist}(u,g)=\|u-g\|_{L^2(\Gamma)}, \quad  Z=Z^*=L^2(Q)
$$ 
with duality pairing 
$$
\langle u,v\rangle_{Z,Z^*} = \int_Q u(\xi)v(\xi)\,d\xi.
$$
Then the empirical loss is given by 
\begin{align*}
\widehat{\mathcal{L}}(u_\theta,v_\tau)
&=
\frac{|\Gamma|}{2N_{\mathrm{bd}}}
\sum_{n=1}^{N_{\mathrm{bd}}}
\big(
\mathcal{B}u_\theta(\xi_n^{\mathrm{bd}})
-
g(\xi_n^{\mathrm{bd}})
\big)^2
\\
&\quad
+
\frac{|Q|}{N_{\mathrm{int}}}
\sum_{n=1}^{N_{\mathrm{int}}}
\big(
\mathcal{F}(\xi_n^{\mathrm{int}},u_\theta(\xi_n^{\mathrm{int}}))
-
f(\xi_n^{\mathrm{int}})
\big)
v_\tau(\xi_n^{\mathrm{int}}).
\end{align*}
The training objective is therefore to compute a saddle point
\[
(\theta^*,\tau^*)
=
\arg\min_\theta \max_\tau
\widehat{\mathcal{L}}(u_\theta,v_\tau).
\]

\subsection{Inf--Sup Solvers for PDEs} 
The resulting training problem is a saddle-point optimization problem. A natural method for solving such problems is the gradient descent--ascent (GDA) algorithm, which alternates between minimizing with respect to $\theta$ and maximizing with respect to $\tau$. 
Under suitable smoothness assumptions and appropriate step-size choices, multi-step ascent--descent schemes can improve stability and convergence behavior \cite{nouiehed_solving_2019}. 

In this work, we adopt an alternating GDA strategy consisting of $k$ ascent steps for the dual variables followed by $l$ descent steps for the primal variables, as summarized in Algorithm~\ref{alg:inf_sup}. The parameters $k$ and $l$ can be tuned empirically depending on the PDE and the network architecture.

\begin{algorithm}[htbp]
\caption{Training Procedure for inf--sup Network}\label{alg:inf_sup}
\begin{flushleft}
\textbf{Input:} 
\begin{itemize}
    \item $N_{int}, N_{bd}$: number of interior and boundary sampling points in $Q$ and $\Gamma$
    \item $k, l$: numbers of gradient descent  and ascent steps
    \item $\theta_0, \tau_0$: initial network parameters
    \item $\eta_\theta^{(0)}, \eta_\tau^{(0)}$: initial learning rates
    \item $T_{\mathrm{decay}}$: learning rate decay interval
    \item $\gamma \in (0,1)$: decay factor
\end{itemize}
\textbf{Initialize:} 
\begin{itemize}
    \item Neural networks $u_{\theta}$, $v_{\tau}$ with parameters $\theta \gets \theta_0$, $\tau \gets \tau_0$
    \item Learning rates $\eta_\theta \gets \eta_\theta^{(0)}$, $\eta_\tau \gets \eta_\tau^{(0)}$
\end{itemize}
\textbf{Training Loop:}
\end{flushleft}
\begin{algorithmic}[1]
\For{each epoch $t=1,2,\dots$}
    \If{$t \bmod T_{\mathrm{decay}} = 0$}, 
        \State update learning rates   $\eta_\theta \gets \gamma \cdot \eta_\theta$
        \State $\eta_\tau \gets \gamma \cdot \eta_\tau$
    \EndIf

    \State Sample interior points $\{\xi_n^{int}\}_{n=1}^{N_{int}} \subset Q$
    \State Sample boundary/initial points $\{\xi_n^{bd}\}_{n=1}^{N_{bd}} \subset \Gamma$

    \For{$l$ steps}
        \State Update the dual network gradient ascent 
        \[
        \tau \gets \tau + \eta_\tau \nabla_{\tau} \widehat{\mathcal{L}}
        \]
    \EndFor

    \For{$k$ steps}
        \State Update the primal network by gradient descent  
        \[
        \theta \gets \theta - \eta_\theta \nabla_{\theta} \widehat{\mathcal{L}}
        \]
    \EndFor
\EndFor
\end{algorithmic}

\begin{flushleft}
\textbf{Output:} Trained  parameters $\theta$, $\tau$ for networks $u_{\theta}$ and $v_{\tau}$.
\end{flushleft}
\end{algorithm}

\section{Error Analysis}

\subsection{Stability and Linearity}
In this section, we focus on linear PDEs, i.e., both  $\mathcal{F}$ and $\mathcal{B}$ are linear operators. While the inf--sup formulation from Section~2 also applies to nonlinear problems, linearity is essential for the error analysis developed here. We assume that the solution satisfies the stability estimate
\begin{align}\label{eq010a}
    \|u\|_X^2 \le C\big( \|f\|_Z^2 + \|g\|_P^2 \big),
\end{align}
where $C$ depends only on the domain $Q$ and the operators $\mathcal{F}$ and $\mathcal{B}$. 

Linearity combined with stability provides a direct  error-control mechanism. Let $u$ denote the exact solution of the PDE and $\tilde u \in X$ an arbitrary approximation (e.g., a neural network). Applying ~\eqref{eq010a} to $u - \tilde u$ then yields
\begin{align}\label{eq010b}
\|u - \tilde u\|_X^2
\;\le\;
C\Big(
\|\mathcal{F}[\xi,\tilde u] - f\|_Z^2
+
\|\mathcal{B}\tilde u - g\|_P^2
\Big).
\end{align} 
Thus, the solution error is fully controlled by the residuals of the PDE and boundary operators. This observation forms the foundational of the  subsequent analysis.

We illustrate these assumptions with two representative examples. 

{\bf Elliptic example.}  Consider the Poisson equation with Dirichlet boundary conditions
\begin{align}\label{eq011}
\begin{split}
-\Delta u &= f, \quad \text{in } \Omega, \\
u &= g, \quad \text{on } \Gamma = \partial \Omega.  
\end{split}
\end{align}
It is well known (See, e.g.,  \cite{schechter_lp_1963}) that if $f \in L^2(\Omega)$ and $g \in H^{\frac{3}{2}}(\partial \Omega)$, then there exists a unique solution $u \in H^2(\Omega)$ satisfying  
\[
\|u\|^2_{H^2(\Omega)} \leq C\left( \|f\|^2_{L^2(\Omega)} + \|g\|^2_{H^{\frac{3}{2}}(\partial \Omega)} \right),
\]
where $C$ depends only on $\Omega$.

{\bf Parabolic example.} Consider the convection--diffusion equation with nonhomogeneous Dirichlet boundary conditions
\begin{align}\label{eq012}
\begin{split}
u_t + \mathbf{b} \cdot \nabla u - \epsilon \Delta u &= f \quad \text{in } \Omega_T, \\
u &= g \quad \text{on } \Sigma := \partial\Omega \times (0,T], \\
u(x,0) &= h(x) \quad \text{on } \Omega.
\end{split}
\end{align}
As shown in \cite{lions_non-homogeneous_1972}, if $f \in L^2(\Omega_T)$, $g \in H^{\frac{3}{2}, \frac{3}{4}}(\Sigma)$, $h \in H^1(\Omega)$, $\mathbf{b}(x,t)$ is smooth, and $\epsilon>0$, then there exists a unique solution $u \in H^{2,1}(\Omega_T)$ to~\eqref{eq012}. Moreover,
\[
\|u\|^2_{H^{2,1}(\Omega_T)} \leq C\left( \|f\|^2_{L^2(\Omega_T)} + \|g\|^2_{H^{\frac{3}{2}, \frac{3}{4}}(\Sigma)} + \|h\|^2_{H^1(\Omega)} \right),
\]
where $C$ depends on $\Omega$, $T$, $\epsilon$, and $\mathbf{b}$. Here,
\begin{align*}
    H^{2,1}(\Omega_T) &:= L^2(0,T; H^2(\Omega)) \cap H^1(0,T; L^2(\Omega)),\\
H^{\frac{3}{2}, \frac{3}{4}}(\Sigma) &:= L^2(0,T; H^{\frac{3}{2}}(\partial\Omega)) \cap H^{\frac{3}{4}}(0,T; L^2(\partial\Omega)).
\end{align*}
\subsection{Error Estimation}
We take $\mathrm{dist}(\mathcal{B}u,g) = \|\mathcal{B}u - g\|_{P}$. Let $(u_n, v_n)$ be the iterate produced after $n$ steps of Algorithm~\ref{alg:inf_sup}, and let $(u,v)$ denote the solution to~\eqref{eq03}, where $u$ also sovles  ~\eqref{eq01}). We aim to bound 
$$
\| u_n - u \|_{X}, \quad \text{and}\quad  \| v_n - v \|_{Z^*} = \| v_n \|_{Z^*}.
$$
Define the residuals:
\begin{itemize}
    \item Interior residual:
    \begin{align}\label{eq013}
        \mathcal{R}_{i,n}(\xi) := \mathcal{F}[\xi, u_n] - f, \quad \text{in } Q.
    \end{align}
    \item Spatial/temporal boundary residual:
    \begin{align}\label{eq014}
        \mathcal{R}_{s,n}(\xi) := \mathcal{B}u_n - g, \quad \text{on } \Gamma.
    \end{align}    
\end{itemize}

\begin{theorem}\label{thm:3}
Let $u \in X$ be the unique solution of ~\eqref{eq01}, and let $(u,v=0)$ solve ~\eqref{eq03} with $\mathrm{dist}(\mathcal{B}u,g) = \|\mathcal{B}u - g\|_{P}$. Let $(u_n,v_n) \in \mathcal{U}_{\theta}\times\mathcal{V}_{\tau} \subset X \times Z^*$ be generated by Algorithm~\ref{alg:inf_sup}. Then
\begin{align}\label{eq015}
\| u_n - u \|^2_{X} + \| v_n - v\|^2_{Z^*} \le C(I_{MC} + I_{NN} + I_{GP}),
\end{align}
where $C$ depends only on $Q$ and $\mathcal{F}$, and 
\begin{align*}
I_{NN} &=  |\mathcal{L}(u_n,\tilde{v}) - \mathcal{L}(u_n,q_{\tilde{v}})| + |\mathcal{L}(\tilde{u},v_n) - \mathcal{L}(w_{\tilde{u}},v_n)|, \\
I_{MC} &= |\mathcal{L}(u_n,q_{\tilde{v}}) - \hat{\mathcal{L}}(u_n,q_{\tilde{v}})| + |\mathcal{L}(w_{\tilde{u}},v_n) - \hat{\mathcal{L}}(w_{\tilde{u}},v_n)|, \\
I_{GP} &= \sup_{q\in\mathcal{V_\tau}} \hat{\mathcal{L}}(u_n,q) - \inf_{w\in\mathcal{U_\theta}} \hat{\mathcal{L}}(w,v_n).
\end{align*}
Here $\tilde{v} = R_Z(F - f)\in Z^*$,  with $R_Z: Z \to Z^*$ the Riesz map satisfying $\langle R_Z z,\,w\rangle_{Z^*,Z}=(z,w)_Z$. Let $\tilde{u}$ solve 
\begin{align*}
\mathcal{F}[\xi, \tilde{u}] - f &= -R^{-1}_{Z}v_n \quad \text{in } Q, \\
\mathcal{B}\tilde{u} &= g \quad \text{on } \Gamma,
\end{align*}
where $R^{-1}_{Z}: Z^* \to Z$ is the inverse Riesz map. Finally,
\[
q_{\tilde{v}} = \arg\min_{q\in \mathcal{V_\tau}}\| q - \tilde{v}\|_{Z^*}, 
\quad 
w_{\tilde{u}} = \arg\min_{w \in \mathcal{U_\theta}}\| w - \tilde{u}\|_{X}.
\]
\end{theorem}

\begin{proof}We  estimate the error using an energy argument.  Let $e := u_n - u$. From the  residual definitions ~\eqref{eq013}, \eqref{eq014} and the PDE~\eqref{eq01}, $e$ satisfies 
\begin{align}\label{eq016}
\begin{split}
\mathcal{F}[\xi, e] &= \mathcal{R}_{i} \quad \text{in } Q, \\
\mathcal{B}e &= \mathcal{R}_{s} \quad \text{on } \Gamma,
\end{split}
\end{align}
where $\mathcal{R}_{i}: = \mathcal{R}_{i,n}$ and $\mathcal{R}_{s}: = \mathcal{R}_{s,n}$. By well-posedness, 
\begin{align}\label{eq017}
\|e\|_X \leq C\left(\|\mathcal{R}_{s}\|_P + \|\mathcal{R}_{i}\|_Z\right).
\end{align}
We next express residual norms in variational form. 
Let $R_Z:Z\to Z^*$ be the Riesz isomorphism so that $\langle R_Z z,\,w\rangle_{Z^*,Z}=(z,w)_Z$. For $u_n \in \mathcal{U}_{\theta}$, define 
$\tilde{v} = R_Z(F - f)\in Z^*$. Then
\begin{align}\label{eq018}
\mathcal{L}(u_n,\tilde{v}) = \tfrac{1}{2}\| \mathcal{B}u_n - g\|^2_{P} + \|F[\xi, u_n] - f\|_Z^2 = \tfrac{1}{2}\|\mathcal{R}_{s}\|^2_P + \|\mathcal{R}_{i}\|^2_Z.
\end{align}
Choosing $q_{\tilde{v}}$ as the best approximation of $\tilde v$ in $\mathcal V_\tau$, i.e.,
\[
q_{\tilde{v}} = \arg\min_{q\in \mathcal{V_\tau}}\| q - \tilde{v}\|_{Z^*},
\]
we have
\begin{align}\label{eq019}
\begin{split}
    \mathcal{L}(u_n,\tilde{v}) 
    &\le \big|\mathcal{L}(u_n,\tilde{v}) - \mathcal{L}(u_n,q_{\tilde{v}})\big|
       + \big|\mathcal{L}(u_n,q_{\tilde{v}}) - \hat{\mathcal{L}}(u_n,q_{\tilde{v}})\big| 
       + \sup_{q\in\mathcal{V_\tau}}\hat{\mathcal{L}}(u_n,q).
\end{split}
\end{align}
Combining~\eqref{eq018} and~\eqref{eq019} yields
\begin{align}\label{eq020}
\tfrac{1}{2}\|\mathcal{R}_{s}\|^2_P + \|\mathcal{R}_{i}\|^2_Z
\le \big|\mathcal{L}(u_n,\tilde{v}) - \mathcal{L}(u_n,q_{\tilde{v}})\big|
   + \big|\mathcal{L}(u_n,q_{\tilde{v}}) - \hat{\mathcal{L}}(u_n,q_{\tilde{v}})\big|
   + \sup_{q\in\mathcal{V_\tau}}\hat{\mathcal{L}}(u_n,q).
\end{align}
For the dual error, well-posedness of~\eqref{eq01}
implies that for each $v_n \in \mathcal{V}_{\tau}$, there exists a unique $\tilde{u} \in X$ solving 
\begin{align*}
\mathcal{F}[\xi, \tilde{u}] - f &= -R^{-1}_{Z}v_n \quad \text{in } Q, \\
\mathcal{B}\tilde{u} &= g \quad \text{on } \Gamma,
\end{align*}
where $R^{-1}_{Z}: Z^* \to Z$ denotes the Riesz inverse.  Then 
\[
\mathcal{L}(\tilde{u},v_n) = -\| v_n\|^2_{Z^*}.
\]
Let $w_{\tilde{u}} = \arg\min_{w \in \mathcal{U_\theta}}\| w - \tilde{u}\|_{X}$. Since $v = 0$, 
\begin{align}\label{eq021}
\| v_n - v\|^2_{Z^*}
\le \big|\mathcal{L}(\tilde{u},v_n) - \mathcal{L}(w_{\tilde{u}},v_n)\big|
   + \big|\mathcal{L}(w_{\tilde{u}},v_n) - \hat{\mathcal{L}}(w_{\tilde{u}},v_n)\big|
   - \inf_{w\in\mathcal{U_\theta}} \hat{\mathcal{L}}(w,v_n).
\end{align}
Combining~\eqref{eq020} and~\eqref{eq021} yields 
\begin{align}\label{eq022}
\tfrac{1}{2}\|\mathcal{R}_s\|^2_{P}
+ \|\mathcal{R}_i\|^2_{Z} 
+ \| v_n - v\|^2_{Z^*}
\le I_{MC} + I_{NN} + I_{GP}.
\end{align}
Together with ~\eqref{eq017}, this proves  Theorem~\ref{thm:3}.
\end{proof}
\textbf{Remark.}
Theorem~\ref{thm:3} shows that the total approximation error of the inf--sup neural network method naturally decomposes into three components: the neural network approximation error $I_{\mathrm{NN}}$, which reflects the expressive power of the hypothesis spaces $\mathcal{U}_\theta$ and $\mathcal{V}_\tau$; the sampling error $I_{\mathrm{MC}}$, arising from the discretization of the continuous loss function via finite interior and boundary sampling; and the optimization gap $I_{\mathrm{GP}}$, which captures the effect of finite-time training of the saddle-point problem. This decomposition clarifies how approximation, numerical integration, and optimization jointly determine the overall accuracy of the method.
\subsection{Inf–Sup Learning under Reduced Regularity} 
In this section, we discuss the practical implementation of the inf–sup method, focusing on the parabolic problem ~\eqref{eq012}, including network construction and loss discretization.

Problem~\eqref{eq012} can be reformulated as the saddle-point problem
\begin{align}\label{eq023}
\inf_{u \in H^{2,1}(\Omega_T)} \sup_{v \in L^2(0,T;L^2(\Omega))} \mathcal{L}(u,v),
\end{align}
where 
\begin{align*}
    \mathcal{L}(u, v) 
    &= \tfrac{1}{2} \|u(x, 0) - h(x)\|_{H^1(\Omega)}^2 
     + \tfrac{1}{2} \|u - g\|_{H^{\frac{3}{2}, \frac{3}{4}}(\Sigma)}^2 \\
    &\quad + \int_{\Omega_T} \big( u_t + \mathbf{b} \cdot \nabla u - \epsilon \Delta u - f \big) v \, dx \, dt.
\end{align*}
A fully discrete realization of this loss is challenging, since the boundary norms involve gradients, whose accurate evaluation requires dense sampling and numerical differentiation, potentially causing instability—especially in high dimensions.
To obtain a more practical scheme, we seek a loss depending only on quantities computable from sampled function values. If the solution admits a stability estimate in lower-regularity spaces, the high-order norms can be replaced by weaker ones without losing theoretical consistency.
Assume therefore that the solution  $u \in X$ of ~\eqref{eq01} satisfies 
\begin{align}\label{eq025aaaa}
    \|u\|^2_{\mathcal{X}} \leq C\big( \|f\|^2_{\mathcal{Z}} + \|g\|^2_{\mathcal{P}} \big),
\end{align}
where $ X \subset \mathcal{X}$, $Z \subset \mathcal{Z}$, and $P \subset \mathcal{P}$ are spaces of lower (or equal) regularity.  Choosing $\mathrm{dist}(\mathcal{B}u,g) = \|\mathcal{B}u - g\|_{\mathcal{P}}$,  we obtain the following result
\begin{theorem}\label{thm:4}
Let $u \in X$ be the unique solution of ~\eqref{eq01} satisfying ~\eqref{eq025aaaa}, and let $(u,v=0)$ solve ~\eqref{eq03} with $\mathrm{dist}(\mathcal{B}u,g) = \|\mathcal{B}u - g\|_{\mathcal{P}}$. Let $(u_n,v_n) \in \mathcal{U}_{\theta}\times\mathcal{V}_{\tau} \subset X \times Z^*$ be generated by Algorithm~\ref{alg:inf_sup}. Then
\begin{align}\label{eq025aaa}
\| u_n - u \|^2_{\mathcal{X}} + \| v_n - v\|^2_{Z^*} \le C(I_{MC} + I_{NN} + I_{GP}),
\end{align}
where the quantities $I_{NN}$, $I_{MC}$, and $I_{GP}$ are defined as in Theorem~\ref{thm:3}.
\end{theorem}
\begin{proof}
The argument follows the same steps as in Theorem~\ref{thm:3},
using the weaker stability estimate ~\eqref{eq025aaaa}. 
\end{proof}
Theorem~\ref{thm:4} shows that if a stability estimate for  $u$ holds in a lower-regularity space, it can be utilized to construct a modified loss with a more tractable discretization. By weakening the regularity requirements, higher-order derivatives need not be evaluated, leading to a simple and more stable numerical implementation. To illustrate this approach, we revisit the parabolic problem~\eqref{eq012} and establish the following lemma, which provide an a priori estimate for $u$ in a weaker norm.
\begin{lemma} \label{lem:LH1-estimate}
Consider the PDE problem~\eqref{eq012}.  
If $f \in L^2(\Omega_T)$, $g \in H^{\frac{3}{2}, \frac{3}{4}}(\Sigma)$, $h \in H^1(\Omega)$, $\mathbf{b}(x,t)$ is smooth, and $\epsilon > 0$,  
then there exists a unique solution $u \in H^{2,1}(\Omega_T)$ to~\eqref{eq012}.  
Moreover, $u$ satisfies the estimate
\begin{align}\label{eq025aa}
    \|u\|_{L^2(\Omega_T)} \leq C\left( \|f\|^2_{L^2(\Omega_T)} + \|g\|^2_{L^2(0,T; H^{\frac{1}{2}}(\partial\Omega))} + \|h\|^2_{L^2(\Omega)} \right),
\end{align}
where the constant $C$ depends on $\Omega$, $T$, $\epsilon$, and $\mathbf{b}$.
\end{lemma}
The proof follows from a standard energy argument and is deferred to the appendix. 

By Theorem~\ref{thm:4} and estimate~\eqref{eq025aa}, the modified saddle point problem reduces to 
\begin{align}\label{eq025b}
\inf_{u \in H^{2,1}(\Omega_T)} \; \sup_{v \in L^2(0,T;L^2(\Omega))} \; 
\mathcal{L}(u,v),
\end{align}
with loss function functional 
\begin{align*}
\mathcal{L}(u,v)
&= \tfrac{1}{2}\|u(x,0)-h(x)\|_{L^2(\Omega)}^2
 + \tfrac{1}{2}\|u-g\|_{L^2(0,T; H^{\frac{1}{2}}(\partial\Omega))}^2 \\
&\quad
 + \int_{\Omega_T}\big(u_t+\mathbf{b}\cdot\nabla u-\epsilon\Delta u-f\big)v\,dx\,dt.
\end{align*}
For numerical implementation, the terms are approximated using a combination of Monte Carlo sampling and tensor–product quadrature based on sampled space–time points.

Specifically, we draw samples as follows: 
\begin{itemize}
\item $m$ space--time samples $\{(x_i,t_i)\}_{i=1}^{m}\subset\Omega_T$ drawn uniformly to approximate integrals over $\Omega_T$,
\item $\hat m$ spatial samples $\{\hat x_i\}_{i=1}^{\hat m}\subset\partial\Omega$ and $\hat n$ temporal samples 
      $\{\hat t_j\}_{j=1}^{\hat n}\subset(0,T]$. 
      The tensor-product set 
      $\{(\hat x_i,\hat t_j)\}_{i\in [\hat m]}^{j\in[\hat n]}
      \subset \Sigma$ is used to approximate the boundary integral appearing in the $L^2(0,T;H^{\frac12}(\partial\Omega))$ term,
\item $m_0$ spatial samples $\{x_i^0\}_{i=1}^{m_0}\subset\Omega$ drawn uniformly to approximate the initial condition term.
\end{itemize}
The space--time integral
\[
\int_{\Omega_T} (\mathcal F[u] - f)\, v \, dx\,dt
\]
and the initial condition term
\[
\|u(x,0)-h(x)\|_{L^2(\Omega)}^2
\]
are approximated using the Monte Carlo estimator
\[
\frac{|\Omega|T}{m}
\sum_{i=1}^{m}
(\mathcal F[u](x_i,t_i)-f(x_i,t_i))\,v(x_i,t_i),
\]
and
\[
\frac{|\Omega|}{m_0}
\sum_{i=1}^{m_0}
|u(x_i^0,0)-h(x_i^0)|^2 .
\]
The boundary term 
\(
\|u-g\|_{L^2(0,T;H^{\frac{1}{2}}(\partial\Omega))}^2
\)
is discretized by 
\begin{itemize}
\item approximating the $L^2$-in-space contributions  via pointwise averages,
\item approximating the fractional $H^{\frac{1}{2}}(\partial\Omega)$ semi-norm using pairwise difference quotients
\[
\frac{|w(\hat x_i)-w(\hat x_k)|^2}{|\hat x_i-\hat x_k|^{d}},
\]
which provides a discrete approximation of the nonlocal boundary norm.
\end{itemize}
Combining the quadrature approximations,  the fully discrete loss functional is 
\begin{align}\label{eq025c}
\begin{split}
\widehat{\mathcal{L}}(u,v)
:=\;&
\frac{|\Omega|}{2m_0}
\sum_{i=1}^{m_0}
|u(x_i^0,0)-h(x_i^0)|^2
\\
&+
\frac{|\partial\Omega|T}{2\hat n\hat m}
\sum_{j=1}^{\hat n}\sum_{i=1}^{\hat m}
|u(\hat x_i,\hat t_j)-g(\hat x_i,\hat t_j)|^2
\\
&+
\frac{|\partial\Omega|^2T}{2\hat n\hat m^2}
\sum_{j=1}^{\hat n}\sum_{i\ne k}^{\hat m}
\frac{|u(\hat x_i,\hat t_j)-g(\hat x_i,\hat t_j)
-(u(\hat x_k,\hat t_j)-g(\hat x_k,\hat t_j))|^2}
{|\hat x_i-\hat x_k|^{d}}
\\
&+
\frac{|\Omega|T}{m}
\sum_{i=1}^{m}
(\mathcal F[u](x_i,t_i)-f(x_i,t_i))\,v(x_i,t_i).
\end{split}
\end{align}
where
\[
\mathcal F[u] = u_t + \mathbf b \cdot \nabla u - \epsilon \Delta u
\]
and $d$ is the spatial dimension.

Algorithm~\ref{alg:inf_sup} can now be applied to train the network. Let $(u,v=0)$ be the solution of the inf--sup problem~\eqref{eq025b}, and let  $(u_n,v_n)$ be the approximation generated by Algorithm~\ref{alg:inf_sup} using  $\widehat{\mathcal{L}}(u, v)$. Then the error satisfies 
\begin{align}\label{eq025d}
\| u_n - u \|^2_{L^2(\Omega_T)} + \| v_n - v\|^2_{L^2(\Omega_T)} 
\le C \big( I_{MC} + I_{NN} + I_{GP} \big),
\end{align}
where  $I_{NN}$, $I_{MC}$, and $I_{GP}$  are defined as in Theorem~\ref{thm:3},  and $C$ depends on $\Omega$, $T$, and $\mathbf{b}$.

\section{Error Bounds}
Having identified the three main sources of error --  neural network approximation error $I_{\text{NN}}$, the Monte Carlo integration error $I_{\text{MC}}$, and the optimization gap $I_{\text{GP}}$. We now aim to establish theoretical bounds for each. 
While \cite{huo_inf-sup_2024} analyzed these components for the elliptic problem~\eqref{eq011}, here we focus on the parabolic problem~\eqref{eq012},  generating approximations $(u_n, v_n)$ via Algorithm~\ref{alg:inf_sup} with the loss $\mathcal{L}(u,v)$ in~\eqref{eq025b} 
and its fully discrete counterpart $\widehat{\mathcal{L}}_s$ in~\eqref{eq025c}.  
Our goal is to quantify the contribution of each component to the total approximation error. 
 
From the previous results, we know that
\[I_{NN} =  |\mathcal{L}(u_n,\tilde{v}) - \mathcal{L}(u_n,q_{\tilde{v}})| + |\mathcal{L}(\tilde{u},v_n) - \mathcal{L}(w_{\tilde{u}},v_n)|.\]
Recall that in the present parabolic setting,
the operator $\mathcal F$ 
is given by 
\[
\mathcal F[u] = u_t + \mathbf b\cdot\nabla u - \epsilon \Delta u.
\]
Hence
\[
\tilde v = \mathcal F[u_n]-f.
\]
The function $\tilde u$ is defined as the solution of
\begin{align*}
\mathcal F[\tilde u] - f &= -v_n, \quad \text{in } \Omega_T,\\
\tilde u &= g, \quad \text{on } \partial\Omega \times (0,T],\\
\tilde u(x,0) &= h(x), \quad \text{in } \Omega .
\end{align*}
Let 
$$
e_v = q_{\tilde{v}} - \tilde{v}, \quad e_u = w_{\tilde{u}} - \tilde{u},
$$
which represent the neural network approximation errors for the dual and
primal variables, respectively. 
Here $q_{\tilde{v}} \in \mathcal{V}_\tau$ denotes  the best approximation of 
$\tilde{v}$ in the neural-network space $\mathcal{V}_\tau$, and 
$w_{\tilde{u}} \in \mathcal{U}_\theta$ denotes the best approximation of 
$\tilde{u}$ in $\mathcal{U}_\theta$. Then, for the first term in $I_{NN}$, we obtain
\begin{align}
\begin{aligned}
|\mathcal L(u_n,\tilde v)-\mathcal L(u_n,q_{\tilde v})|
&= \left|
\int_{\Omega_T}
(\mathcal F[u_n]-f)\,(\tilde v-q_{\tilde v})
\,dx\,dt
\right|\\
&\le\|\tilde v\|_{L^2(\Omega_T)}
\|\tilde v-q_{\tilde v}\|_{L^2(\Omega_T)}\\
&= \|\tilde v\|_{L^2(\Omega_T)}
\|e_v\|_{L^2(\Omega_T)} .
\end{aligned}
\label{eq025}
\end{align}
For the second term, we compute
\begin{align*}
&\left|\mathcal{L}(w_{\tilde{u}}, v_n) - \mathcal{L}(\tilde{u}, v_n)\right| \\
&= \Bigg| 
\frac{1}{2} \left(
\|w_{\tilde{u}} - g\|^2_{L^2(0,T; H^{\frac{1}{2}}(\partial\Omega))} 
+ \|w_{\tilde{u}}(x,0) - h\|^2_{L^2(\Omega)}
\right) 
+ \int_{\Omega_T}
(\mathcal F[w_{\tilde u}]-\mathcal F[\tilde u])\,v_n
\,dx\,dt
\Bigg| \\
& =\Bigg|  
\frac{1}{2} \left(\|e_u\|^2_{L^2(0,T; H^{\frac{1}{2}}(\partial\Omega))} 
+ \|e_u(\cdot, 0)\|^2_{L^2(\Omega)}\right) +   \int_{\Omega_T} 
 \mathcal F[e_u]
\cdot v_n \, dx\,dt\Bigg|. 
\end{align*}
Applying the trace theorem and Cauchy-Schwartz inequality gives 
\begin{align*} 
\left|\mathcal{L}(w_{\tilde{u}}, v_n) - \mathcal{L}(\tilde{u}, v_n)\right|\le 
C_{tr} \|e_u\|^2_{H^{2,1}(\Omega_T)} 
+ \left\|\mathcal F[e_u] \right\|_{L^2(\Omega_T)} 
\cdot \|v_n\|_{L^2(\Omega_T)},
\end{align*} 
where $C_{tr}>0$ is a constant that depends only on the domain $\Omega$. Next, using the definition of $\mathcal F$,
$$
\left\|\mathcal F[e_u] \right\|_{L^2(\Omega_T)} \le \left( 1 + \|\mathbf{b}\|_{L^\infty(\Omega_T)} + \epsilon \right)  
\|e_u\|_{H^{2,1}(\Omega_T)}.  
$$
Therefore 
\begin{align}\label{eq026}
    \left|\mathcal{L}(w_{\tilde{u}}, v_n) - \mathcal{L}(\tilde{u}, v_n)\right| 
 \leq 
 C \left(\|e_u\|_{H^{2,1}(\Omega_T)} + \|v_n\|_{L^2(\Omega_T)} 
\right) 
\|e_u\|_{H^{2,1}(\Omega_T)}, 
\end{align}
where $C$ depends on the trace constant and $\|\mathbf{b}\|_{L^\infty(\Omega_T)}$. 
Combining (\ref{eq025}) and (\ref{eq026}) yields,
\begin{align}\label{eq21}
I_{NN} \le C_v\| e_v\|_{L^2(\Omega_T)} + C_u\|e_u\|_{H^{2,1}(\Omega_T)}, 
\end{align}
where $C_v = \| \tilde{v} \|_{L^2(\Omega_T)}$ and $C_u = C
\left(\|e_u\|_{H^{2,1}(\Omega_T)} + \|v_n\|_{L^2(\Omega_T)}\right).$
Thus, $I_{NN}$ is small whenever the network classes $\mathcal{U}_\theta$ and $\mathcal{V}_\tau$ accurately approximates  $\tilde{u}$ and $\tilde{v}$. 
Recent approximation results in Sobolev-type spaces (e.g., \cite{abdeljawad_approximations_2022}) further support this claim.  

Next, consider the Monte Carlo sampling error
\[
I_{\mathrm{MC}} = \left|\mathcal{L}(u_n, q_{\tilde{v}}) - \hat{\mathcal{L}_s}(u_n, q_{\tilde{v}})\right| 
+ \left|\mathcal{L}(w_{\tilde{u}}, v_n) - \hat{\mathcal{L}_s}(w_{\tilde{u}}, v_n)\right|.
\]
It suffices to establish a bound for the generic sampling error
$\big|\mathcal{L}(u,v)-\widehat{\mathcal{L}}_s(u,v)\big|.$ Once such an estimate is obtained, the same bound applies directly to both components in the definition of $I_{\mathrm{MC}}$.

Based on the definitions of $\mathcal{L}$ and $\widehat{\mathcal{L}}_s$ in \eqref{eq025b}-\eqref{eq025c}, the sampling error can be decomposed as
\[
\big|\mathcal{L}(u,v)-\widehat{\mathcal{L}}_s(u,v)\big|
\le D_1(u) + D_2(u) + D_3(u,v),
\] 
where each term corresponds to one component of the loss:
\begin{align*}
D_1(u) &= 
\Big|
\tfrac{1}{2}\|u(x,0)-h\|_{L^2(\Omega)}^2
-
\frac{|\Omega|}{2m_0}
\sum_{i=1}^{m_0}
|u(x_i^0,0)-h(x_i^0)|^2
\Big|,
\\[0.6ex]
D_2(u) &= 
\Big|
\tfrac{1}{2}\|u-g\|_{L^2(0,T;H^{\frac{1}{2}}(\partial\Omega))}^2
-
\text{(the discrete boundary approximation in \eqref{eq025c})}
\Big|,
\\[0.6ex]
D_3(u,v) &= 
\Big|
\int_{\Omega_T}(\mathcal F[u]-f)v\,dx\,dt
-
\frac{|\Omega|T}{m}
\sum_{i=1}^{m}
\big(\mathcal F[u](x_i,t_i)-f(x_i,t_i)\big)
v(x_i,t_i)
\Big|.
\end{align*}
Hence, it suffices to estimate $D_1$, $D_2$, and $D_3$ separately.

By standard Monte Carlo theory (see \cite{caflisch_monte_1998}), if the spatial
samples $\{x_i^0\}_{i=1}^{m_0}$ are drawn i.i.d. uniformly 
from $\Omega$, then for sufficiently large $m_0$,
\[
D_1(u)
\;\le\; |\Omega|\,\sigma_1\,m_0^{-\frac12},
\]
almost surely, where 
\[
\sigma_1^2
= \int_{\Omega}
\left(
(u(x,0)-h(x))^2
- \frac{1}{|\Omega|}\int_{\Omega}(u(x',0)-h(x'))^2\,dx'
\right)^2 dx.
\]
Similarly, if the space--time samples are drawn i.i.d. uniformly from  $\Omega_T$, then for sufficiently large $m$, 
\[
D_3(u,v)
\;\le\; |\Omega_T|\,\sigma_2\,m^{-\frac12},
\]
almost surely, where
\[
\sigma_2^2
= \int_{\Omega_T}
\left(
(\mathcal F[u] - f)v
- \frac{1}{|\Omega_T|}
   \int_{\Omega_T}(\mathcal F[u] - f)v\,dx'\,dt'
\right)^2 dx\,dt.
\]
We first estimate $\sigma_1^2$. Define
\[
w(x) := (u(x,0)-h(x))^2,
\qquad
\bar w := \frac{1}{|\Omega|}\int_{\Omega} w(x')\,dx'.
\]
Then
\[
\sigma_1^2
=
\int_{\Omega} (w-\bar w)^2\,dx.
\]
Using the identity
\[
\int_{\Omega} (w-\bar w)^2\,dx
=
\int_{\Omega} w^2\,dx
-
|\Omega|\,\bar w^2,
\]
we obtain the upper bound
\[
\sigma_1^2
\le
\int_{\Omega} w^2\,dx
=
\|u(\cdot,0)-h\|_{L^4(\Omega)}^4.
\]
Therefore,
\[
D_1(u)
\le
|\Omega|\,\sigma_1\,m_0^{-\frac12}
\le
|\Omega|\,\|u(\cdot,0)-h\|_{L^4(\Omega)}^2\,m_0^{-\frac12}.
\]
Next, we estimate $\sigma_2^2$. Similar to the estimate for $\sigma_1^2$, using the identity over $\Omega_T$ and dropping the non-positive mean squared term yields the upper bound
\[
\sigma_2^2
\le
\int_{\Omega_T}\big((\mathcal F[u]-f)v\big)^2\,dx\,dt
=
\|(\mathcal F[u]-f)v\|_{L^2(\Omega_T)}^2.
\]
Therefore,
\[
D_3(u,v)
\le
|\Omega_T|\,\sigma_2\,m^{-\frac12}
\le
|\Omega_T|\,\|(\mathcal F[u]-f)v\|_{L^2(\Omega_T)}\,m^{-\frac12}.
\]
Using H\"older's inequality, we further obtain
\[
D_3(u,v)
\le
|\Omega_T|\,\|v\|_{L^\infty(\Omega_T)}
\|\mathcal F[u]-f\|_{L^2(\Omega_T)}\,m^{-\frac12}.
\]
The term $D_2$ involves  fractional Sobolev norms on $\partial\Omega_T$. To estimate it, we introduce a tensor--product grid $G^{r,r'}_{k,k'} \subset \overline{\Omega}_T$, with mesh sizes
\[
h = 2^{-k}(r-1)^{-1}, 
\qquad
h' = 2^{-k'}(r'-1)^{-1},
\]
We define $\bar G^{r,r'}_{k,k'} := G^{r,r'}_{k,k'} \cap \Sigma = \{(\hat x_i,\hat t_j)\}_{i\in [\hat m]}^{j\in[\hat n]}$, the lateral boundary grid of size
\[
\hat{m}\,\hat{n}
= 2d \,(r2^k)^{\,d-1} (r'2^{k'}).
\]
Assume $u-g \in \operatorname{Tr}\!\left(H^p(0,T;H^s(\Omega))\right)$ 
with $s > \frac{d}{2}$, $p > \frac12$, and choose 
\[
r > \max\{s,1\}, 
\qquad 
r' > \max\{p,1\}.
\]
Standard approximation estimates for fractional Sobolev norms (see  \cite{bonito_convergence_2025}) then yield
\[
D_2(u) \le C_d\,\hat{m}^{-\alpha}\,\hat{n}^{-\beta},
\]
where $C_d>0$ depends only on  $\|u-g\|_{\operatorname{Tr}(H^p(0,T;H^s(\Omega)))}$ and is independent of $\hat{m}$ and $\hat{n}$. The exponents $\alpha,\beta>0$ depend on the regularity of $u-g$ and  the spatial dimension $d$; in particular 
\[
\alpha = \frac{s-1}{d-1},
\qquad
\beta = p.
\]
To derive the upper bound for $I_{\mathrm{MC}}$, we evaluate the generic sampling error bound for the two pairs $(u_n, q_{\tilde{v}})$ and $(w_{\tilde{u}}, v_n)$. We first consider the pair $(u_n, q_{\tilde{v}})$. Substituting $u = u_n$ and $v = q_{\tilde{v}}$, where $\tilde{v} = \mathcal F[u_n]-f$, we obtain
\begin{align*}
    D_1(u_n) &\le |\Omega|\|u_n(\cdot,0)-h\|^2_{L^4(\Omega)}\,m_0^{-\frac12},\\
    D_3(u_n, q_{\tilde{v}})
    &\le |\Omega_T| \|q_{\tilde{v}}\|_{L^\infty(\Omega_T)} \|\tilde{v}\|_{L^2(\Omega_T)} m^{-\frac12}.
\end{align*}
Next, for the pair $(w_{\tilde{u}}, v_n)$, we set $u = w_{\tilde{u}} = \tilde{u} + e_u$ and $v = v_n$. Using the identity $\mathcal F[\tilde u]-f = -v_n$ and the linearity of $\mathcal F$, we obtain
\[
    \mathcal F[w_{\tilde{u}}] - f = \mathcal F[e_u] + \mathcal F[\tilde{u}] - f = \mathcal F[e_u] - v_n.
\]
It follows that
\begin{align*}
    D_3(w_{\tilde{u}}, v_n)
    &\le |\Omega_T| \|v_n\|_{L^\infty(\Omega_T)} \|\mathcal F[e_u] - v_n\|_{L^2(\Omega_T)}\,m^{-\frac12} \\
    &\le |\Omega_T| \|v_n\|_{L^\infty(\Omega_T)} \left( \|\mathcal F[e_u]\|_{L^2(\Omega_T)} + \|v_n\|_{L^2(\Omega_T)} \right) m^{-\frac12}.
\end{align*}
Since $\mathcal F$ is a linear parabolic operator ($\mathcal F[u] = u_t + \mathbf b\cdot\nabla u - \epsilon \Delta u$), there exists a constant $C_{\mathcal F} > 0$ such that $\|\mathcal F[e_u]\|_{L^2(\Omega_T)} \le C_{\mathcal F}\|e_u\|_{H^{2,1}(\Omega_T)}$. Thus,
\[
    D_3(w_{\tilde{u}}, v_n) \le |\Omega_T| \|v_n\|_{L^\infty(\Omega_T)} \left( C_{\mathcal F}\|e_u\|_{H^{2,1}(\Omega_T)} + \|v_n\|_{L^2(\Omega_T)} \right) m^{-\frac12}.
\]
Similarly, the $D_1$ term satisfies
\[
    D_1(w_{\tilde{u}})\le |\Omega|\|w_{\tilde{u}}(\cdot,0)-h\|^2_{L^4(\Omega)}\,m_0^{-\frac12}.
\]
Summing the contributions from $D_1, D_2,$ and $D_3$ for both terms, we obtain
\begin{align*}
I_{\mathrm{MC}} 
\;  &\le\;
C_1\,m_0^{-\frac12} + C_2\,\|e_u\|_{H^{2,1}(\Omega_T)}m^{-\frac12} + C_3\,m^{-\frac12} 
+ C_4\,\hat{m}^{-\alpha} \hat{n}^{-\beta},
\end{align*}
where the specific constants are given by:
\begin{align*}
C_1 &= \ |\Omega|\Big(\|u_n(\cdot,0)-h\|^2_{L^4(\Omega)} + \|w_{\tilde{u}}(\cdot,0)-h\|^2_{L^4(\Omega)} \Big), \\[1ex]
C_2 &= |\Omega_T| C_{\mathcal F} \|v_n\|_{L^\infty(\Omega_T)}, \\[1ex]
C_3 &= |\Omega_T| \Big( \|q_{\tilde{v}}\|_{L^\infty(\Omega_T)}\|\tilde{v}\|_{L^2(\Omega_T)} + \|v_n\|_{L^\infty(\Omega_T)}\|v_n\|_{L^2(\Omega_T)} \Big), \\[1ex]
C_4 &= C_d(u_n) + C_d(w_{\tilde{u}}).
\end{align*}
Here, $C_d(u_n)$ and $C_d(w_{\tilde{u}})$ correspond to the fractional Sobolev trace approximation constants from the $D_2$ bounds for $u_n$ and $w_{\tilde{u}}$, respectively.

We summarize the above results in the following theorem. 
\begin{theorem}\label{thm:7}
Let $\Omega \subset \mathbb{R}^d$ be a bounded domain with sufficiently smooth boundary $\partial\Omega$, and let 
$u \in H^{2,1}(\Omega_T)$ be the unique solution of \eqref{eq012},
with $h \in H^1(\Omega) \cap L^4(\Omega)$ and $f \in L^2(\Omega_T)$. Assume further that $u \in H^p\!\left(0,T;\,H^s(\Omega)\right)$ with $s > \frac{d}{2}$ and $p > \frac{1}{2}$. 
Let $(u_n,v_n)\in \mathcal{U}_\theta\times\mathcal{V}_\tau$ be the iterates generated 
after $n$ steps of Algorithm~1. We assume that the parameter spaces defining the neural networks $\mathcal{U}_\theta$ and $\mathcal{V}_\tau$ are uniformly bounded. Furthermore, we assume $\mathcal{U}_\theta$ employs at least $C^{\max(p,s)}(\mathbb{R})$ activation function, and $\mathcal{V}_\tau$ employs an activation function of at least $C^0(\mathbb{R})$. Let $I_{GP}$ be defined as in Theorem~\ref{thm:3}.  Then, almost surely,  
\begin{equation}\label{eq24}
\|u_n - u\|_{L^2(\Omega_T)}^2 
+ \|v_n - v\|_{L^2(\Omega_T)}^2
\;\le\; C\,\mathcal{E},
\end{equation}
where the total error $\mathcal{E}$ is 
\begin{align*}
\mathcal{E}
&=
C_1\, m_0^{-\frac12}
+ C_v \,\|e_v\|_{L^2(\Omega_T)}
+ \big(C_u + C_2 m^{-\frac12}\big)\,\|e_u\|_{H^{2,1}(\Omega_T)}
\\
&\qquad
+\, C_3\,m^{-\frac12}
+ C_4\,\hat{m}^{-\alpha}\hat{n}^{-\beta}
+ I_{GP}.
\end{align*}
Here $C, C_u, C_v, C_1 \dots C_4$ are constants defined in 
the preceding analysis, and $\alpha = \frac{s-1}{d-1}$ and $\beta = p$.
\end{theorem}
\textbf{Remark 1.} The uniform boundedness of the error constants $C_u, C_v$ and $C_1, \dots, C_4$ as $n \to \infty$ is a direct consequence of the structural assumptions placed on the neural networks. By restricting the parameter spaces of both $\mathcal{U}_\theta$ and $\mathcal{V}_\tau$ to compact sets, we prevent the weights from diverging to infinity during the inf--sup optimization. For the dual network $v_n$, this parameter bound, combined with a standard $C^0(\mathbb{R})$ activation function, is entirely sufficient to keep the constants $C_v, C_u, C_2$ and $C_3$ finite as $n \to \infty$. For the primal network $u_n$, a $C^{\max(p,s)}(\mathbb{R})$ activation would suffice to keep the constants $C_1$ and $C_4$ finite as $n \to \infty$.

\textbf{Remark 2.} The training loss gap $ I_{GP} $ in (\ref{eq24}) depends on both the optimization algorithm  and the structure of the empirical loss $\hat{\mathcal{L}}$. In practice, $ I_{GP} $ typically decreases with the number of iterations; see \cite{nemirovski_prox-method_2004, lin_gradient_2020} under specific structural assumptions.  
In general nonconvex–nonconcave settings, convergence depends on additional structural assumptions; for instance, \cite{jin_what_2020} shows that gradient descent–ascent (GDA) can converge to a local saddle point of 
$\hat{\mathcal{L}}$.

\textbf{Remark 3.} The total error in ~(\ref{eq24}) is governed by three main factors: the approximation capacity of the neural networks $ \mathcal{U}_\theta $ and $ \mathcal{V}_\tau $, the total number of sampling points $ N $, and the effectiveness of the training algorithm. Specifically:
\begin{itemize}
    \item The terms involving $ \|e_u\|_{H^{2,1}(\Omega_T)} $ and $ \|e_v\|_{L^2(\Omega_T)} $ reflect the intrinsic approximation error determined by the expressive power of the chosen network architectures.
    \item The terms with negative powers of sampling numbers arise from the empirical approximation of the loss function. 
    \item The gap term $ I_{GP} $ reflects the optimization error induced by the training procedure.
\end{itemize}
Therefore, the overall error can be reduced by enhancing the network expressiveness (e.g., increasing depth or width, or adopting more suitable  architectures), increasing the number of sampling points $N$, and employing more effective optimization algorithms.

\textbf{Remark 4.} Theorem~\ref{thm:7} provides theoretical justification of the inf--sup neural network framework applied to a class of linear convection--diffusion equations ~\eqref{eq012}. Although the present analysis is restricted to linear problems, the numerical experiments in  Subsections~6.2--6.4 demonstrate strong empirical performance on several nonlinear PDEs. These results suggest  that the inf--sup formulation remains robust beyond the linear setting.
\section{Numerical Experiments}
 
In this section, we present  numerical experiments to assess the performance of the proposed inf--sup neural network framework for solving both linear and nonlinear parabolic PDEs. For each example, we report the evolution of the empirical loss and its individual components, as well as the absolute and relative $L^2$ error of the neural network solution.  Logarithmic plots are provided  illustrate convergence behavior, along with visual comparisons between the exact and predicted solutions. We also perform systematic validation studies to investigate the effects of key hyperparameters, including network depth and width, spatial dimensionality, and the number of interior and boundary sampling points. The results demonstrate the accuracy, stability, and robustness of the method across different spatial dimensions. Both the primal network $u_\theta$ and the dual network $v_\tau$ are fully connected feedforward neural networks with $\tanh$ activation functions, trained using the RMSprop optimizer \cite{noauthor_tieleman_nodate}, according to the procedure outlined in Algorithm~\ref{alg:inf_sup}.

\subsection{
Convection--Diffusion Equation}

We consider the linear convection--diffusion equation
\begin{align*}
u_t + \mathbf{b}\cdot\nabla u &= \beta\,\Delta u, 
\quad \text{in } \Omega_T, \\
u(\mathbf{x},t) &= g(\mathbf{x},t), 
\quad \text{on } \Sigma,\\
u(\mathbf{x},0) &= f(\mathbf{x}), 
\quad \text{in } \Omega.
\end{align*}
Here,  $\Omega = [0,2\pi]^d$ and  $T = 1$. 
We set $\beta = 10^{-4}$ and choose a  constant velocity 
$
\mathbf{b} = \mathbf{1}_d = (1,\dots,1)\in\mathbb{R}^d.
$
The data are prescribed as 
\[
f(\mathbf{x}) = \sin\!\left(\sum_{i=1}^d x_i\right),
\qquad
g(\mathbf{x},t) = e^{-\beta d t}
\sin\!\left(\sum_{i=1}^d x_i - d t\right),
\]
yielding the exact solution  
\[
u(\mathbf{x}, t) =  e^{-\beta d t}
\sin\!\left(\sum_{i=1}^d x_i - d t\right).
\]
The proposed inf--sup neural network is tested for $d=3$ and $d=5$. The model is trained for $20{,}000$ epochs in three dimensions and $45{,}000$ epochs in five dimensions to ensure convergence. In both cases, $6{,}400$ uniformly sampled space--time collocation points are used for training, together with $640$ boundary points on $\partial\Omega\times(0,T]$ and $640$ initial points on $\Omega\times\{t=0\}$. 
The primal and dual networks $u_\theta$ and $v_\tau$ share the same architecture, consisting of four hidden layers with $64$ neurons per layer. Optimization starts with a learning rate of $10^{-3}$, which is hlaved every $5{,}000$ epochs. 

To visualize high-dimensional solutions, we adopt a two-dimensional slice representation: the learned approximation $u_\theta$ is evaluated on the $(x_1,x_2)$ plane while the remaining spatial coordinates are fixed. This enables qualitative comparison with the exact solution  without loss of interpretability as the spatial dimension increases.

We assess the proposed inf--sup neural network for the linear convection--diffusion equation in dimensions $d=3$ and $d=5$. 
For each case, we report training dynamics, error convergence, and qualitative comparisons with the exact solution. A systematic validation study further examines the influence of dimensionality, sampling density, and network architecture. We also analyze the training behavior of the dual network $v_\tau$.

{\bf Training dynamics and error convergence ($d=3$)}.\\
Figure~\ref{fig:d3-training-linear} shows the evolution of the empirical loss, including its  interior, boundary, and initial components, together with the absolute and relative $L^2$ errors. 
All loss terms decrease steadily,  indicating that the method simultaneously enforces the PDE, boundary, and initial conditions in a stable manner. After $20{,}000$ epochs, the relative $L^2$ error reaches approximately $6\times 10^{-3}$, with an absolute $L^2$ error of about $4\times 10^{-3}$.

\begin{figure}[!htbp]
    \centering
    \begin{subfigure}[t]{0.32\textwidth}
        \centering
        \includegraphics[width=\textwidth]{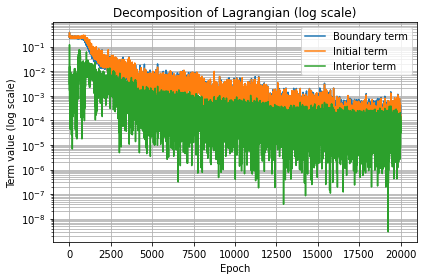}
        \caption{Loss decomposition.}
    \end{subfigure}\hfill
    \begin{subfigure}[t]{0.32\textwidth}
        \centering
        \includegraphics[width=\textwidth]{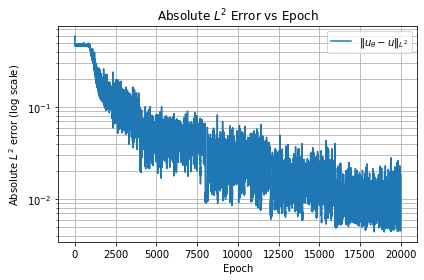}
        \caption{$L^2$ error.}
    \end{subfigure}\hfill
    \begin{subfigure}[t]{0.32\textwidth}
        \centering
        \includegraphics[width=\textwidth]{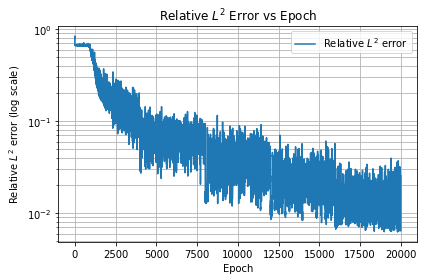}
        \caption{Relative $L^2$ error.}
    \end{subfigure}
    \caption{Training dynamics for the convection--diffusion equation in $d=3$.}
    \label{fig:d3-training-linear}
\end{figure}

{\bf Qualitative solution comparison ($d=3$).}
To visualize the learned solution, we consider two-dimensional slices obtained by varying $(x_1,x_2)\in[0,2\pi]^2$ while fixing $x_3=\pi$, evaluated at the final time $T=1$. 
Figure~\ref{fig:d3-solution} displays the exact solution, the neural network approximation, and the corresponding pointwise absolute error. The learned solution accurately reproduces the spatial structure, with errors that remain small and smoothly distributed. across the domain. 

\begin{figure}[H]
    \centering
    \includegraphics[width=0.95\textwidth]{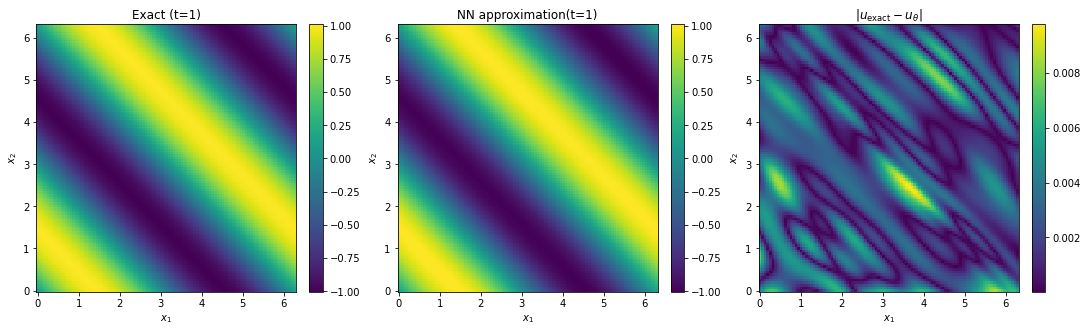}
    \caption{
    Solution comparison for the $d=3$ convection--diffusion equation at $t=1$.
    From left to right: exact solution $u(\cdot,1)$, neural network approximation $u_\theta(\cdot,1)$, and pointwise absolute error $|u-u_\theta|$.
    The solutions are evaluated on the $(x_1,x_2)$ slice with $x_3=\pi$.
    }
    \label{fig:d3-solution}
\end{figure}

{\bf Training dynamics and error convergence ($d=5$).}
Figure~\ref{fig:d5-training-linear} shows the training curves for the five higher-dimensional case. 
As expected, convergence is slower in higher dimensions under the same network architecture and sampling strategy. Nevertheless, after $45{,}000$ epochs, the method achieves a relative $L^2$ error of $7.54\times 10^{-3}$ and an absolute $L^2$ error of $5.35\times 10^{-3}$.

\begin{figure}[H]
    \centering
    \begin{subfigure}[t]{0.32\textwidth}
        \centering
        \includegraphics[width=\textwidth]{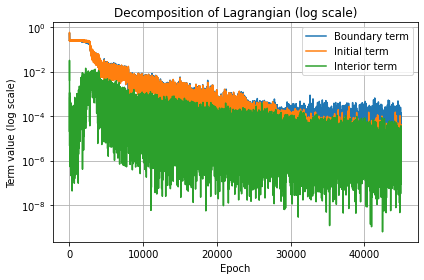}
        \caption{Loss decomposition.}
    \end{subfigure}\hfill
    \begin{subfigure}[t]{0.32\textwidth}
        \centering
        \includegraphics[width=\textwidth]{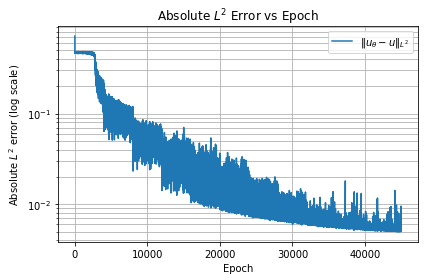}
        \caption{$L^2$ error.}
    \end{subfigure}\hfill
    \begin{subfigure}[t]{0.32\textwidth}
        \centering
        \includegraphics[width=\textwidth]{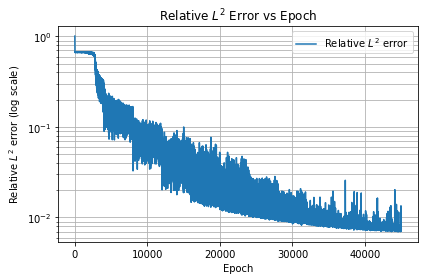}
        \caption{Relative $L^2$ error.}
    \end{subfigure}
    \caption{Training dynamics for the convection--diffusion equation in $d=5$.}
    \label{fig:d5-training-linear}
\end{figure} 

{\bf Qualitative solution comparison ($d=5$).}
For the five dimensional case, we again use a two-dimensional slice representation by varying $(x_1,x_2)\in[0,2\pi]^2$ while fixing $x_3=x_4=x_5=\pi$, evaluated  at $T=1$. As shown in Figure~\ref{fig:d5-solution},  the learned solution remains accurate,  and the pointwise absolute error is uniformly small across the slice. 
\begin{figure}[!htbp]
    \centering
\includegraphics[width=0.95\textwidth]{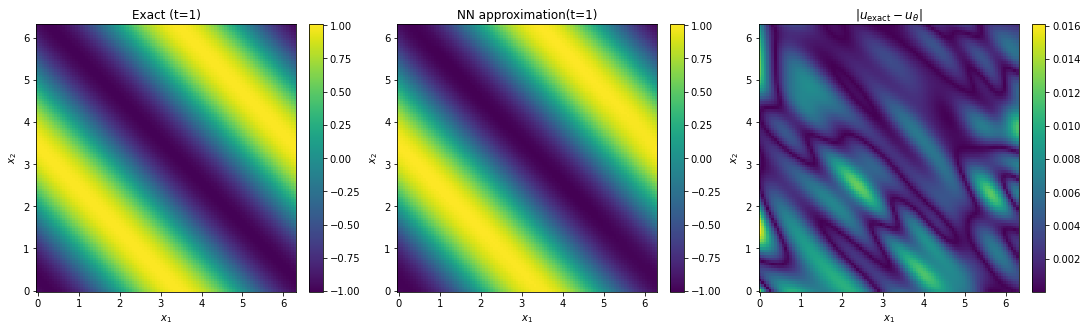}
    \caption{
    Solution comparison for the $d=5$ convection--diffusion equation at $t=1$.
    From left to right: exact solution $u(\cdot,1)$, neural network approximation $u_\theta(\cdot,1)$, and pointwise absolute error $|u-u_\theta|$.
    The solutions are evaluated on the $(x_1,x_2)$ slice while fixing $x_3=x_4=x_5=\pi$.
    }
    \label{fig:d5-solution}
\end{figure}

{ \bf Validation studies.}\\
We assess robustness through four validation experiments summarized in Figure~\ref{fig:validation}, investigating the effects of (i) sampling density, (ii) spatial dimension, (iii) network depth, and (iv) network width. 
The results show that an intermediate sampling density yields the best convergence for a fixed training budget. Higher spatial dimensions slow convergence, deeper networks enhance it, and increasing width improves performance up to a point, beyond which larger models may require additional training or more data to fully realize their capacity.

\begin{figure}[!htbp]
    \centering
    \begin{subfigure}[t]{0.48\textwidth}
        \centering
        \includegraphics[width=\textwidth]{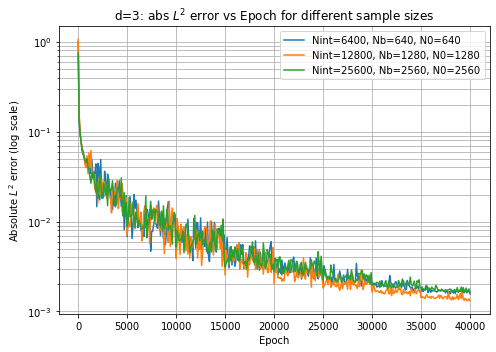}
        \caption{Sampling density.}
    \end{subfigure}\hfill
    \begin{subfigure}[t]{0.48\textwidth}
        \centering
        \includegraphics[width=\textwidth]{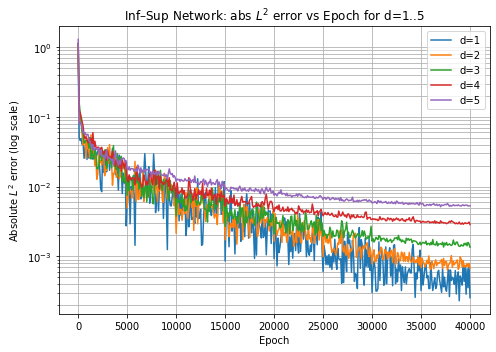}
        \caption{Dimension $d=1,\dots,5$.}
    \end{subfigure}

    \vspace{0.6em}

    \begin{subfigure}[t]{0.48\textwidth}
        \centering
        \includegraphics[width=\textwidth]{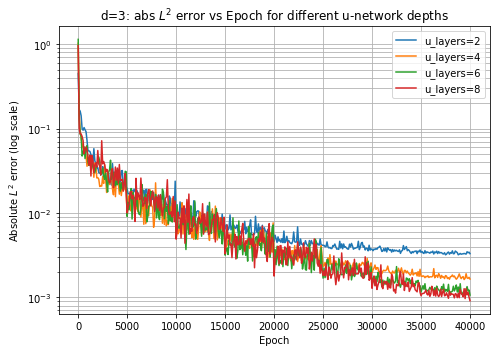}
        \caption{Depth (layers).}
    \end{subfigure}\hfill
    \begin{subfigure}[t]{0.48\textwidth}
        \centering
        \includegraphics[width=\textwidth]{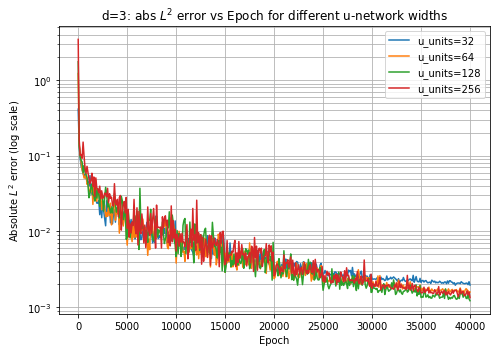}
        \caption{Width (neurons per layer).}
    \end{subfigure}
    \caption{Validation studies.}
    \label{fig:validation}
\end{figure}

{\bf Behavior of the dual network.}
Finally, Figure~\ref{fig:vtau} shows the evolution of the $L^2$ norm of the dual network $v_\tau$ (log scale) for $d=3$ over $40{,}000$ epochs. The norm decreases rapidly during early training and subsequently oscillates near zero, reflecting  the saddle-point structure: the dual variable acts as an adversarial witness enforcing the PDE constraint and remains small once the primal network accurately approximates the solution. 

\begin{figure}[H]
    \centering
    \includegraphics[width=0.62\textwidth]{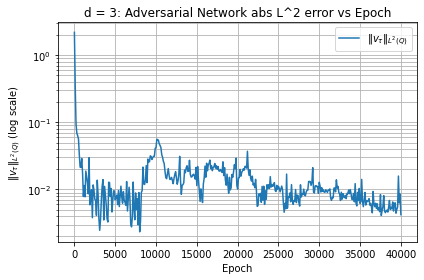}
    \caption{Evolution of the adversarial network $v_\tau$ in $d=3$ ($L^2$ norm over $40{,}000$ epochs).}
    \label{fig:vtau}
\end{figure}
Overall, the numerical experiments demonstrate that the proposed inf--sup neural network provides stable training, accurate approximations, and robust performance across varying spatial dimensions, sampling densities, and network architectures.

\subsection{
Nonlinear Reaction-Diffusion}
In this subsection, we present experimental results for the nonlinear reaction-diffusion equation. Although a rigorous theoretical analysis of the proposed inf--sup network in the nonlinear setting has not yet been established, our experiments indicate that the method effectively produces accurate solutions. 

We consider the equation:
\begin{align*}
u_t - \Delta u - u^2 - h &= 0, \quad \text{in } \Omega_T, \\
u(\mathbf{x},t) &= g(\mathbf{x}, t), \quad \text{on } \Sigma,\\
u(\mathbf{x},0) &= f(\mathbf{x}), \quad \text{in } \Omega, 
\end{align*}
where the spatial domain is $\Omega = [0,1]^d$ and the final time  is $T = 1$. The source term $h$, Dirichlet boundary condition $g$, and initial condition $f$ are defined by 
\[\begin{cases}
    h(\mathbf{x}, t) =  \left( \frac{\pi}{2} \right)^d \left( \left( \frac{\pi^2d}{2} - 2 \right) e^{-t} \prod_{i=1}^d \sin \left( \frac{\pi}{2}x_i + \frac{\pi i}{2} \right) - 4\left(\frac{\pi}{2} \right)^de^{-2t} \prod_{i=1}^d \sin^2 \left( \frac{\pi}{2}x_i + \frac{\pi i}{2} \right) \right), \\[10pt]
    g(\mathbf{x}, t) =  \left( \frac{\pi}{2} \right)^d 2e^{-t} \prod_{i=1}^d \sin \left( \frac{\pi}{2}x_i + \frac{\pi i}{2} \right), \\[10pt]
    f(\mathbf{x}) =  \left( \frac{\pi}{2} \right)^d 2 \prod_{i=1}^d \sin \left( \frac{\pi}{2}x_i + \frac{\pi i}{2} \right).
\end{cases}\]
It can be readily verified that the solution is 
\[
u(\mathbf{x}, t) = \left( \frac{\pi}{2} \right)^d 2e^{-t} \prod_{i=1}^d \sin \left( \frac{\pi}{2}x_i + \frac{\pi i}{2} \right).
\]
For $d=5$, the model is trained for $35{,}000$ epochs using $6{,}400$ interior points, $640$ boundary points, and $640$ initial points.
Both the primal and dual networks consist of four hidden layers with $64$ neurons per layer.  
For $d=20$, the training budget is increased to $80{,}000$ epochs, with $12{,}800$ interior points, $1{,}280$ boundary points, and $1{,}280$ initial points. The network architecture is correspondingly expanded eight hidden layers with $128$ neurons per layer. 
In all cases, training is performed using the RMSprop optimizer with an initial learning rate of $10^{-3}$ and stepwise decay. 

Figure~\ref{fig:rd-d5-training} illustrates the evolution of the empirical loss  and its individuala components, along with the absolute and relative $L^2$ errors for the $d=5$ case.
All components of the loss --interior residual, boundary mismatch, and initial condition term -- decay steadily toward zero as training progresses. After $35{,}000$ epochs, the relative $L^2$ error attains approximately $7.73\times 10^{-3}$, while the absolute $L^2$ error is about $1.65\times 10^{-2}$.

\begin{figure}[H]
    \centering
    \begin{subfigure}[t]{0.32\textwidth}
        \centering
        \includegraphics[width=\textwidth]{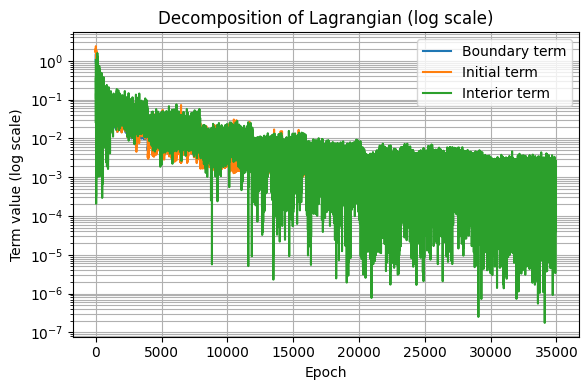}
        \caption{Loss decomposition.}
    \end{subfigure}\hfill
    \begin{subfigure}[t]{0.32\textwidth}
        \centering
        \includegraphics[width=\textwidth]{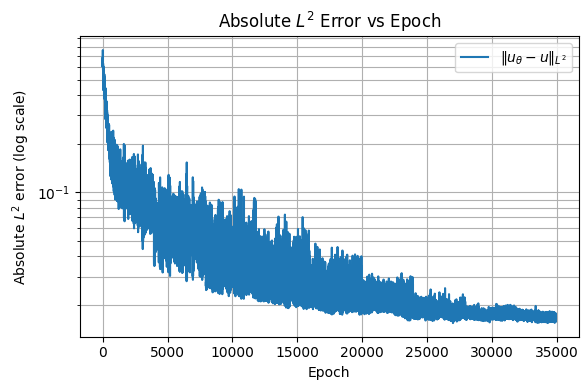}
        \caption{$L^2$ error.}
    \end{subfigure}\hfill
    \begin{subfigure}[t]{0.32\textwidth}
        \centering
        \includegraphics[width=\textwidth]{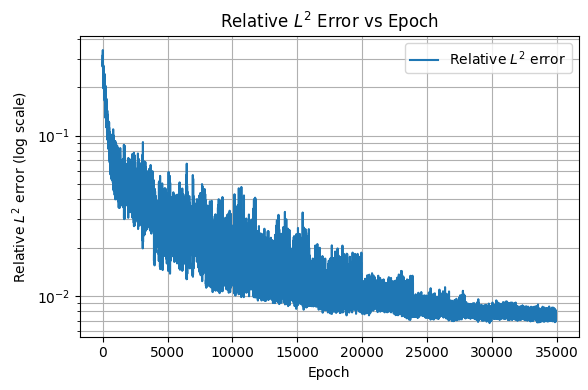}
        \caption{Relative $L^2$ error.}
    \end{subfigure}
    \caption{Training dynamics for the nonlinear reaction--diffusion equation in $d=5$.}
    \label{fig:rd-d5-training}
\end{figure}

To visualize the learned solution, we evaluate $u_\theta$ on a two-dimensional slice by varying $(x_1,x_2)\in[0,2\pi]^2$ while fixing $x_3=\cdots=x_5=0.5$ at the final time $t=1$.
Figure~\ref{fig:rd-d5-solution} compares the exact solution, the neural network approximation, and the corresponding pointwise absolute error. The results demonstrate that the neural network accurately reproduces the spatial structure of the solution, with errors that are small in magnitude and smoothly distributed across the domain. 
\begin{figure}[!htbp]
    \centering
    \includegraphics[width=0.95\textwidth]{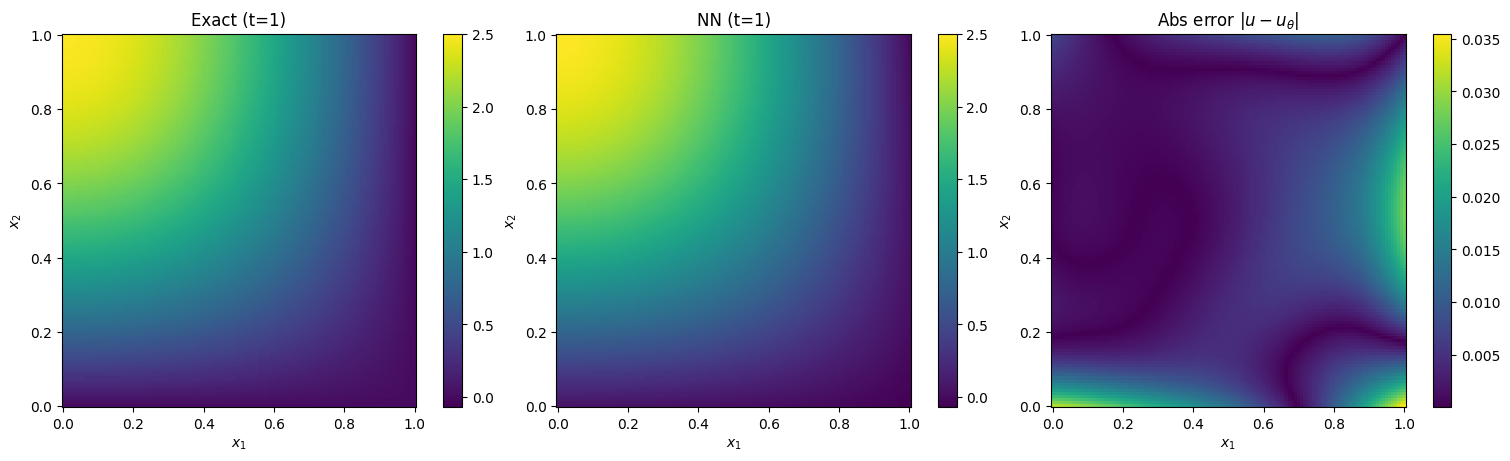}
    \caption{
    Solution comparison for the nonlinear reaction--diffusion equation in $d=5$ at $t=1$.
    From left to right: exact solution, neural network approximation, and pointwise absolute error.
    The solution is evaluated on the $(x_1,x_2)$ slice while fixing $x_3=\cdots=x_5=0.5$.
    }
    \label{fig:rd-d5-solution}
\end{figure}

Figure~\ref{fig:rd-d20-training} and Figure~\ref{fig:rd-d20-solution} present the analogous results for the case $d=20$.
Despite the significantly inreqased dimensionality and nonlinearity, the proposed inf--sup network remains stable and  exhibits clear convergence behavior. All loss components decrease steadily toward zero, and both the absolute and relative $L^2$ errors decay consistently through training.  
After $80{,}000$ epochs, the relative $L^2$ error reaches approximately $5.62\times 10^{-3}$, while the absolute $L^2$ error is reduced to about $2.26\times 10^{-2}$.

\begin{figure}[H]
    \centering
    \begin{subfigure}[t]{0.32\textwidth}
        \centering
        \includegraphics[width=\textwidth]{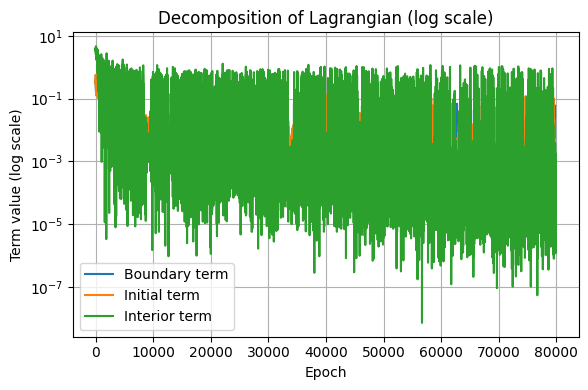}
        \caption{Loss decomposition.}
    \end{subfigure}\hfill
    \begin{subfigure}[t]{0.32\textwidth}
        \centering
        \includegraphics[width=\textwidth]{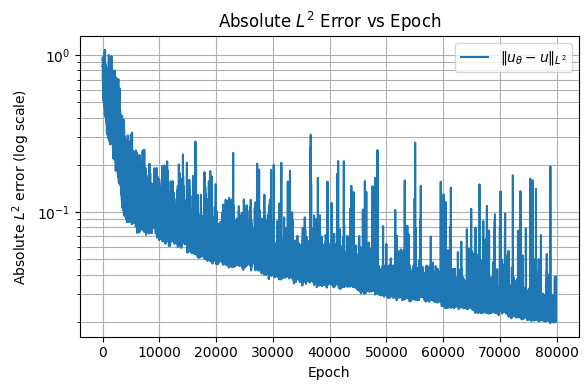}
        \caption{$L^2$ error.}
    \end{subfigure}\hfill
    \begin{subfigure}[t]{0.32\textwidth}
        \centering
        \includegraphics[width=\textwidth]{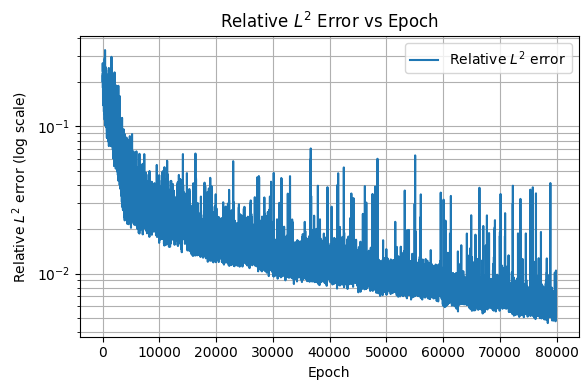}
        \caption{Relative $L^2$ error.}
    \end{subfigure}
    \caption{Training dynamics for the nonlinear reaction--diffusion equation in $d=20$.}
    \label{fig:rd-d20-training}
\end{figure}

\begin{figure}[H]
    \centering
    \includegraphics[width=0.95\textwidth]{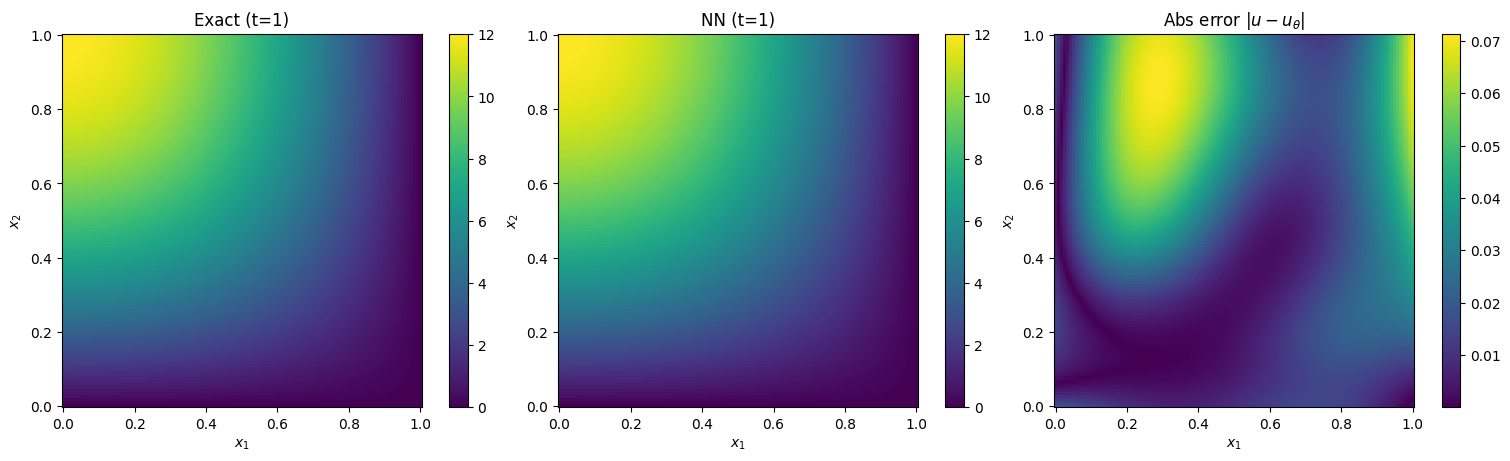}
    \caption{
    Solution comparison for the nonlinear reaction--diffusion equation in $d=20$ at $t=1$.
    From left to right: exact solution, neural network approximation, and pointwise absolute error.
    The solution is evaluated on the $(x_1,x_2)$ slice with $x_3=\cdots=x_{20}=0.5$.
    }
    \label{fig:rd-d20-solution}
\end{figure}

\subsection{
One-Dimensional Viscous Burgers' and Allen--Cahn Equations}

In this subsection, we assess the
performance of the 
proposed inf--sup neural network on two nonlinear benchmark problems 
characterized by sharp gradients and pronounced ramp–cliff structures: the viscous Burgers' equation and the Allen--Cahn equation. These test cases are designed  to evaluate the capability of the method to accurately resolve localized steep transitions and nonlinear dynamics. 

We first consider the viscous Burgers' equation subject to the initial condition and homogeneous Dirichlet boundary conditions: 
\begin{align*}
u_t + u\,u_x - \frac{0.01}{\pi}\,u_{xx} &= 0,
\quad x \in [-1,1],\; t \in (0,1], \\
u(x,0) &= -\sin(\pi x),
\quad x \in [-1,1], \\
u(-1,t) &= u(1,t) = 0,
\quad t \in [0,1].
\end{align*}
Next, we consider the Allen--Cahn equation with initial and periodic boundary conditions: 
\begin{align*}
u_t - 10^{-4} u_{xx} + 5u^3 - 5u &= 0,
\quad x \in [-1,1],\; t \in (0,1], \\
u(x,0) &= x^2 \cos(\pi x),
\quad x \in [-1,1], \\
u(-1,t) &= u(1,t), \\
u_x(-1,t) &= u_x(1,t),
\quad t \in (0,1].
\end{align*}
For both problems, we use $12{,}800$ interior collocation points, $1{,}280$ boundary points, and $1{,}280$ initial-condition samples. The primal and dual networks each comprise four hidden layers with $64$ neurons per layer. Training is carried out using the  RMSprop optimizer with an initial learning rate of $10^{-3}$. 

For the viscous Burgers' equation, the model is trained for $40{,}000$ epochs,  with learning rate decayed every $5{,}000$ epochs. For the Allen--Cahn equation, training is extended to $100{,}000$ epochs,  with learning rate decay applied every $10{,}000$ epochs.

For the one-dimensional viscous Burgers’ equation, Figure~\ref{fig:burgers-training} illustrates the evolution of the empirical loss and its decomposition into interior, boundary, and initial components, along with the absolute and relative $L^2$ errors during training. All loss components  decrease steadily toward zero, indicating that the inf--sup formulation effectively enforces the PDE residual as well as the boundary and initial conditions during the optimization process. 

After $40{,}000$ training epochs, the relative $L^2$ error is reduced to approximately $5.6\times 10^{-3}$, while the absolute $L^2$ error reaches about $2.8\times 10^{-3}$. Despite the pronounced  ramp–cliff structure characteristic of the viscous Burgers’ solution, the neural network successfully  resolves the steep gradients without generating spurious oscillations. These results demonstrate the stability and robustness of the proposed method
in capturing sharp transition layers. 

\begin{figure}[H]
    \centering
    \begin{subfigure}[t]{0.32\textwidth}
        \centering
        \includegraphics[width=\textwidth]{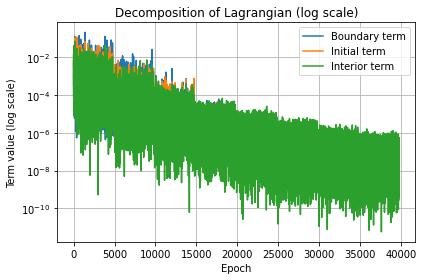}
        \caption{Loss decomposition.}
    \end{subfigure}\hfill
    \begin{subfigure}[t]{0.32\textwidth}
        \centering
        \includegraphics[width=\textwidth]{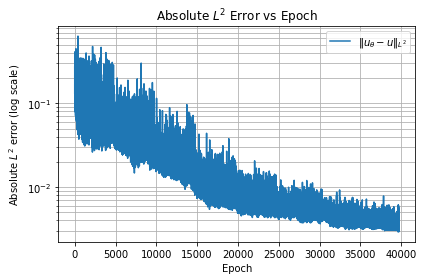}
        \caption{$L^2$ error.}
    \end{subfigure}\hfill
    \begin{subfigure}[t]{0.32\textwidth}
        \centering
        \includegraphics[width=\textwidth]{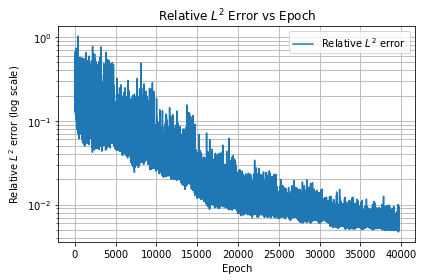}
        \caption{Relative $L^2$ error.}
    \end{subfigure}
    \caption{Training dynamics for the one-dimensional viscous Burgers' equation.}
    \label{fig:burgers-training}
\end{figure}

To further evaluate the approximation accuracy, Figure~\ref{fig:burgers-solution} presents a comparison between the exact solution and the neural network approximation at $t=1$. The predicted solution exhibits excellent agreement with the reference solution across the entire spatial domain, accurately resolving both the smooth regions and the sharp transition layer without noticeable distortion.

\begin{figure}[!htbp]
    \centering
    \includegraphics[width=0.7\textwidth]{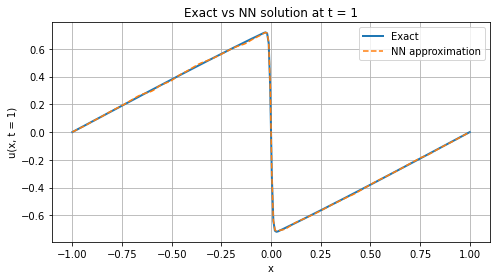}
    \caption{
    Exact solution and neural network approximation of the viscous Burgers' equation at $t=1$.
    }
    \label{fig:burgers-solution}
\end{figure}

Figure ~\ref{fig:ac-training} and Figure ~\ref{fig:ac-solution} summarize the numerical results for the Allen--Cahn equation. As in  the Burgers case, the interior, boundary, and initial components of the empirical loss all decrease steadily during training, indicating that the inf–sup formulation enforces the nonlinear PDE constraint in a stable and consistent manner.

Owning to the stiff cubic reaction term, a longer training horizon is required. Nevertheless, the training process remains stable, without noticeable oscillations or divergence in either the primal or dual variables. After $100{,}000$ epochs, the relative $L^2$ error reaches approximately $8.62\times 10^{-3}$, with an absolute $L^2$ error of about $6.16\times 10^{-3}$.

The comparison at $t=1$ further shows that the neural network accurately captures the phase-transition structure and the smooth interfacial layers characteristic of the Allen--Cahn equation dynamics. The pointwise errors remain small and uniformly distributed across the domain, indicating that the inf--sup network effectively resolves  nonlinear reaction mechanisms,  even in the presence of stiffness.  
\begin{figure}[!htbp]
    \centering
    \begin{subfigure}[t]{0.32\textwidth}
        \centering
        \includegraphics[width=\textwidth]{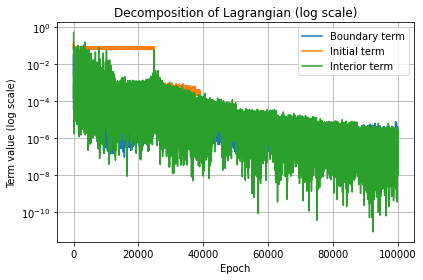}
        \caption{Loss decomposition.}
    \end{subfigure}\hfill
    \begin{subfigure}[t]{0.32\textwidth}
        \centering
        \includegraphics[width=\textwidth]{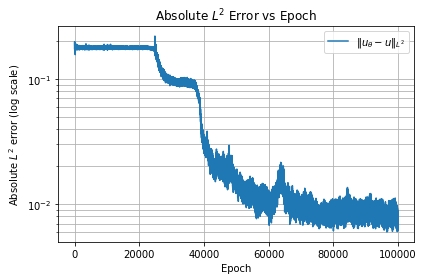}
        \caption{$L^2$ error.}
    \end{subfigure}\hfill
    \begin{subfigure}[t]{0.32\textwidth}
        \centering
        \includegraphics[width=\textwidth]{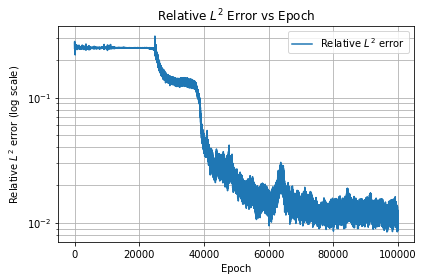}
        \caption{Relative $L^2$ error.}
    \end{subfigure}
    \caption{Training dynamics for the one-dimensional Allen--Cahn equation.}
    \label{fig:ac-training}
\end{figure}

\begin{figure}[H]
    \centering
    \includegraphics[width=0.7\textwidth]{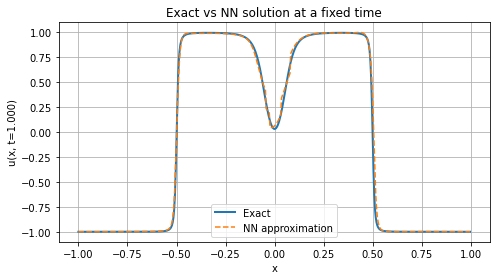}
    \caption{
    Exact solution and neural network approximation of the Allen--Cahn equation at $t=1$.
    }
    \label{fig:ac-solution}
\end{figure}
These one-dimensional experiments demonstrate that the inf--sup neural network effectively resolves  nonlinear convection, diffusion, and stiff reaction mechanisms. Even in the presence of ramp–cliff structures and sharp transition layers, the method achieves stable training, consistent loss decay, and accurate solution approximation. Combined with the higher-dimensional experimental results presented earlier, these findings further confirm the robustness of the inf--sup framework across a broad class of linear and nonlinear PDEs.

\section{Discussion and Outlook} 
In this work, we extend the inf--sup neural network framework introduced in \cite{huo_inf-sup_2024} for elliptic problems to a broader class of PDEs, thereby enhancing both its generality and practical applicability. The proposed methodology follows a two-stage strategy. First, the original PDE is reformulated as a saddle-point (inf--sup) optimization problem, and we establish its theoretical equivalence for classical solutions under suitable assumptions. Second, the resulting min--max problem is solved using two fully connected neural networks: a primal network approximating the PDE solution and a dual  network representing the associated Lagrange multiplier, trained via an alternating gradient descent--ascent.

From a theoretical standpoint, we derive an error decomposition of the form 
\[
\text{Total Error} \;\le\; I_{\mathrm{NN}} + I_{\mathrm{MC}} + I_{\mathrm{GP}},
\]
where $I_{\mathrm{NN}}$ denotes the neural network approximation error, $I_{\mathrm{MC}}$ the Monte Carlo integration error induced by empirical discretization, and $I_{\mathrm{GP}}$ the optimization gap associated with the saddle-point training procedure. 
We show that each component can be controlled under reasonable conditions, including sufficient expressive capacity of the neural networks in appropriate Sobolev-type space--time norms, adequate sampling density for Monte Carlo approximation, and a suitable optimization strategy for the inf--sup problem.

The effectiveness of the proposed framework is further
 demonstrated through  numerical experiments on time-dependent parabolic equations in both linear and nonlinear regimes, confirming that the inf--sup neural network delivers stable and accurate approximations.

Despite these encouraging results, several limitations of the current framework remain. The analysis assumes the existence and uniqueness of sufficiently regular PDE solutions, and the error estimates are primarily derived for linear problems. In addition, numerical performance depends on the choice of boundary norms and  the efficiency of the saddle-point optimization algorithm, introducing trade-offs between theoretical accuracy and practical implementation.

Future work should address several key directions. These include refining the inf–sup formulation via enhanced regularization and alternative boundary enforcement strategies, extending the framework to PDEs with weak or low-regularity solutions, and developing rigorous error analyses for nonlinear operators. On the computational side, improving optimization algorithms tailored to saddle-point structures is essential for reducing the optimization gap and accelerating convergence. Collectively,  these efforts aim to strengthen the theoretical foundations of inf--sup neural networks while enhancing their robustness and efficiency for a broad class of complex PDEs.


\bibliography{ref}
\bibliographystyle{amsplain}

\appendix

\renewcommand{\theequation}{\thesection.\arabic{equation}}
\section{Proof of Lemma~\ref{lem:LH1-estimate}}
\begin{proof}
We employ the method of transposition. Let $\psi \in L^2(\Omega_T)$ be an arbitrary test function. Consider the backward parabolic (adjoint) problem:
\begin{align}
    -\phi_t - \epsilon \Delta \phi - \nabla \cdot (\mathbf{b} \phi) &= \psi \quad \text{in } \Omega_T, \label{eqA1} \\
    \phi &= 0 \quad \text{on } \partial \Omega_T, \nonumber \\
    \phi(x,T) &= 0 \quad \text{in } \Omega. \nonumber
\end{align}

Since the terminal data is zero and the coefficients are smooth, standard parabolic regularity theory ensures that there exists a unique strong solution $\phi \in H^{2,1}(\Omega_T)$. Furthermore, there exists a constant $C > 0$ (depending on $\Omega, T, \epsilon$, and $\mathbf{b}$) such that
\[
    \|\phi\|_{H^{2,1}(\Omega_T)} \le C \|\psi\|_{L^2(\Omega_T)}.
\]
Because $\phi \in H^{2,1}(\Omega_T)$, we have the continuous embedding $\phi \in C([0,T]; H^1_0(\Omega))$, which gives
\[
    \|\phi(\cdot, 0)\|_{L^2(\Omega)} \le C \|\psi\|_{L^2(\Omega_T)}.
\]
Additionally, the trace theorem yields that the normal derivative on the boundary satisfies
\[
    \left\| \frac{\partial \phi}{\partial \mathbf{n}} \right\|_{L^2(0,T; H^{\frac{1}{2}}(\partial\Omega))} \le C \|\phi\|_{L^2(0,T; H^2(\Omega))} \le C \|\psi\|_{L^2(\Omega_T)}.
\]

Now, we multiply the original PDE $u_t + \mathbf{b} \cdot \nabla u - \epsilon \Delta u = f$ by $\phi$ and integrate over $\Omega_T$:
\[
    \int_{\Omega_T} (u_t + \mathbf{b} \cdot \nabla u - \epsilon \Delta u) \phi \, dx \, dt = \int_{\Omega_T} f \phi \, dx \, dt.
\]
We apply integration by parts (Green's identities) to shift all derivatives onto $\phi$. 
For the time derivative, using $\phi(x,T) = 0$ and $u(x,0) = h(x)$, we have
\[
    \int_{\Omega_T} u_t \phi \, dx \, dt = - \int_\Omega h(x) \phi(x,0) \, dx - \int_{\Omega_T} u \phi_t \, dx \, dt.
\]
For the advection term, using $\phi = 0$ on $\partial \Omega_T$,
\[
    \int_{\Omega_T} (\mathbf{b} \cdot \nabla u) \phi \, dx \, dt = - \int_{\Omega_T} u \nabla \cdot (\mathbf{b} \phi) \, dx \, dt.
\]
For the diffusion term, applying Green's second identity and using $u = g$ and $\phi = 0$ on $\partial \Omega_T$,
\[
    -\epsilon \int_{\Omega_T} (\Delta u) \phi \, dx \, dt = - \epsilon \int_{\Omega_T} u \Delta \phi \, dx \, dt + \epsilon \int_{\partial\Omega_T} g \frac{\partial \phi}{\partial \mathbf{n}} \, ds \, dt.
\]

Combining these and substituting them into the integral equation yields
\begin{align*}
    \int_{\Omega_T} u \Big( -\phi_t - \nabla \cdot (\mathbf{b} \phi) - \epsilon \Delta \phi \Big) \, dx \, dt &= \int_{\Omega_T} f \phi \, dx \, dt + \int_\Omega h(x) \phi(x,0) \, dx \\
    &\quad - \epsilon \int_{\partial\Omega_T} g \frac{\partial \phi}{\partial \mathbf{n}} \, ds \, dt.
\end{align*}
Using the fact that $\phi$ solves the adjoint problem~\eqref{eqA1}, the bracketed term on the left-hand side simplifies perfectly to $\psi$. Thus,
\[
    \int_{\Omega_T} u \psi \, dx \, dt = \int_{\Omega_T} f \phi \, dx \, dt + \int_\Omega h(x) \phi(x,0) \, dx - \epsilon \int_{\partial\Omega_T} g \frac{\partial \phi}{\partial \mathbf{n}} \, ds \, dt.
\]

Taking the absolute value and applying the Cauchy--Schwarz inequality (and the duality pairing on the boundary), we obtain
\begin{align*}
    \left| \int_{\Omega_T} u \psi \, dx \, dt \right| &\le \|f\|_{L^2(\Omega_T)} \|\phi\|_{L^2(\Omega_T)} + \|h\|_{L^2(\Omega)} \|\phi(\cdot,0)\|_{L^2(\Omega)} \\
    &\quad + \epsilon \|g\|_{L^2(0,T; H^{\frac{1}{2}}(\partial\Omega))} \left\| \frac{\partial \phi}{\partial \mathbf{n}} \right\|_{L^2(0,T; H^{-\frac12}(\partial\Omega))}.
\end{align*}
Since $H^{\frac{1}{2}}(\partial\Omega) \subset H^{-\frac12}(\partial\Omega)$, we can bound the normal derivative trace strictly by $C\|\psi\|_{L^2(\Omega_T)}$ as established earlier. Applying the bounds for $\phi$ gives
\[
    \left| \int_{\Omega_T} u \psi \, dx \, dt \right| \le C \left( \|f\|_{L^2(\Omega_T)} + \|h\|_{L^2(\Omega)} + \|g\|_{L^2(0,T; H^{\frac{1}{2}}(\partial\Omega))} \right) \|\psi\|_{L^2(\Omega_T)}.
\]
Because this inequality holds for any arbitrary test function $\psi \in L^2(\Omega_T)$, we conclude by Riesz representation/duality that $u \in L^2(\Omega_T)$ and satisfies
\[
    \|u\|_{L^2(\Omega_T)} \le C \left( \|f\|_{L^2(\Omega_T)} + \|g\|_{L^2(0,T; H^{\frac{1}{2}}(\partial\Omega))} + \|h\|_{L^2(\Omega)} \right).
\]
This establishes the estimate.
\end{proof}

\end{document}

%% file: math_commands.tex
\def\eqref#1{equation~\ref{#1}}

\def\1{\bm{1}}

\DeclareMathAlphabet{\mathsfit}{\encodingdefault}{\sfdefault}{m}{sl}
\SetMathAlphabet{\mathsfit}{bold}{\encodingdefault}{\sfdefault}{bx}{n}

